\documentclass[12pt]{article}
\usepackage{layout}
\usepackage[makeroom]{cancel}
\usepackage{bbm}
\usepackage[affil-it]{authblk}

\usepackage{mathtools}
\usepackage{comment}

\usepackage[shortlabels]{enumitem}
\setlist[enumerate,1]{wide, labelindent=0pt,label={\upshape(\roman*)}}

\mathtoolsset{showonlyrefs}
\usepackage[normalem]{ulem} 

\usepackage{amsfonts,dsfont}
\usepackage{amsmath}
\usepackage{amssymb,bm}
\usepackage{amsthm}
\usepackage{graphicx} \usepackage{multicol}
\usepackage{mathrsfs} \usepackage[all,cmtip]{xy}
\usepackage[rgb,dvipsnames]{xcolor}
\usepackage{float}
\usepackage{ulem}
\usepackage{circuitikz}

\usepackage[colorlinks,linkcolor=blue, urlcolor  = blue,citecolor=blue,pagebackref,hypertexnames=false, breaklinks]{hyperref}

\usepackage{cite}

\usepackage[nottoc,notlot,notlof]{tocbibind}

\usepackage{float}

\allowdisplaybreaks[4]

\newlength{\bibitemsep}
\newlength{\bibparskip}
\let\oldthebibliography\thebibliography
\renewcommand\thebibliography[1]{%
 \oldthebibliography{#1}%
 \setlength{\parskip}{\bibitemsep}%
 \setlength{\itemsep}{\bibparskip}%
}

\newtheorem{theorem}{Theorem}[section]

\newtheorem{definition}[theorem]{Definition}
\newtheorem{proposition}[theorem]{Proposition}

\newtheorem{corollary}[theorem]{Corollary}
\newtheorem*{theorem*}{Theorem}
\newtheorem{lemma}[theorem]{Lemma}
\newtheorem{remark}[theorem]{Remark}
\newtheorem{example}[theorem]{Example}
\newtheorem{examples}[theorem]{Examples}
\newtheorem{foo}[theorem]{Remarks}

\newtheorem{conjecture}[theorem]{Open question}
\newtheorem{open question}[theorem]{Open Question}
\newtheorem{c/p}[theorem]{Conjecture/Proposition}

\newcommand{\dom}{\operatorname{dom}}

\def\vint{\mathop{\mathchoice%
 {\setbox0\hbox{$\displaystyle\intop$}\kern 0.22\wd0%
 \vcenter{\hrule width 0.6\wd0}\kern -0.82\wd0}%
 {\setbox0\hbox{$\textstyle\intop$}\kern 0.2\wd0%
 \vcenter{\hrule width 0.6\wd0}\kern -0.8\wd0}%
 {\setbox0\hbox{$\scriptstyle\intop$}\kern 0.2\wd0%
 \vcenter{\hrule width 0.6\wd0}\kern -0.8\wd0}%
 {\setbox0\hbox{$\scriptscriptstyle\intop$}\kern 0.2\wd0%
 \vcenter{\hrule width 0.6\wd0}\kern -0.8\wd0}}%
 \mathopen{}\int}

\newcommand{\eng}{\mathcal{E}}

\usepackage[textwidth=3.2cm]{todonotes}

\title{Heat kernel gradient estimates  for the Vicsek set}
\author{Fabrice Baudoin\footnote{F.B. was partly funded by the NSF grant DMS-2247117 when most of this research was conducted.}, Li Chen\footnote{L.C. is partly funded by the Simons Foundation Travel Support for Mathematicians Grant \#853249.}}

\begin{document}

\maketitle

\begin{abstract}
   We prove pointwise and $L^p$ gradient estimates for the heat kernel on the bounded and unbounded Vicsek set and applications to Sobolev inequalities are given. We also define a Hodge semigroup in that setting and prove estimates for its kernel.
\end{abstract}
\tableofcontents

\section{Introduction}

Heat kernel gradient estimates have proven to be a  powerful and versatile tool in many different situations. For instance, they play an important role in the abstract Bakry-\'Emery theory of curvature bounds and functional inequalities, see \cite{BGL}. They can also be used in harmonic analysis to prove boundedness of Riesz transforms, see \cite{ACDH,ChenCoulhon,RieszSubelliptic} and in Riemannian and sub-Riemannian geometry to prove Sobolev and isoperimetric inequalities, see \cite{Baudoinsurvey, BaudoinKim, varopoulos , Ledouximproved}. On a complete Riemannian manifold, if we denote by $p_t$ the heat kernel, $d$ the Riemannian distance, and $\mu(B(x,\sqrt{t}))$ the Riemannian volume measure of a geodesic ball with center $x$ and radius $\sqrt{t}$, then  the Gaussian heat kernel estimates
\begin{align}\label{HKEintro}
 \frac{ c_1 }{\mu(B(x,\sqrt t))} \exp \left(-c_2\frac{d(x,y)^2}{t}\right) \le p_t(x,y) \le \frac{ c_3 }{\mu(B(x,\sqrt t))} \exp \left(-c_4\frac{d(x,y)^2}{t}\right)
\end{align}
are equivalent to the combination of the volume doubling property and the 2-Poincar\'e inequality, see the celebrated works \cite{grigodoubling,saloff}. However the matching gradient estimate
\begin{align}\label{GBintro}
| \nabla_x p_t (x,y) |\le  \frac{ C }{\sqrt{t}\mu(B(x,\sqrt t))} \exp \left(-c\frac{d(x,y)^2}{t} \right)
\end{align}
does not always follow from \eqref{HKEintro} and requires different assumptions on the manifold, like the Ricci curvature being non-negative, see \cite{LiYau}. 
The heat kernel Gaussian estimates \eqref{HKEintro} and corresponding gradient estimates \eqref{GBintro} can more generally be studied in metric measure spaces and Dirichlet spaces equipped with a carr\'e du champ operator $\Gamma(f)$ which is used as a substitute for $|\nabla f |^2$, see for instance \cite{gradientCoulhon} and references therein.

\begin{figure}[htb]\centering
\includegraphics[trim={60 90 180 200},clip,height=0.15\textwidth]{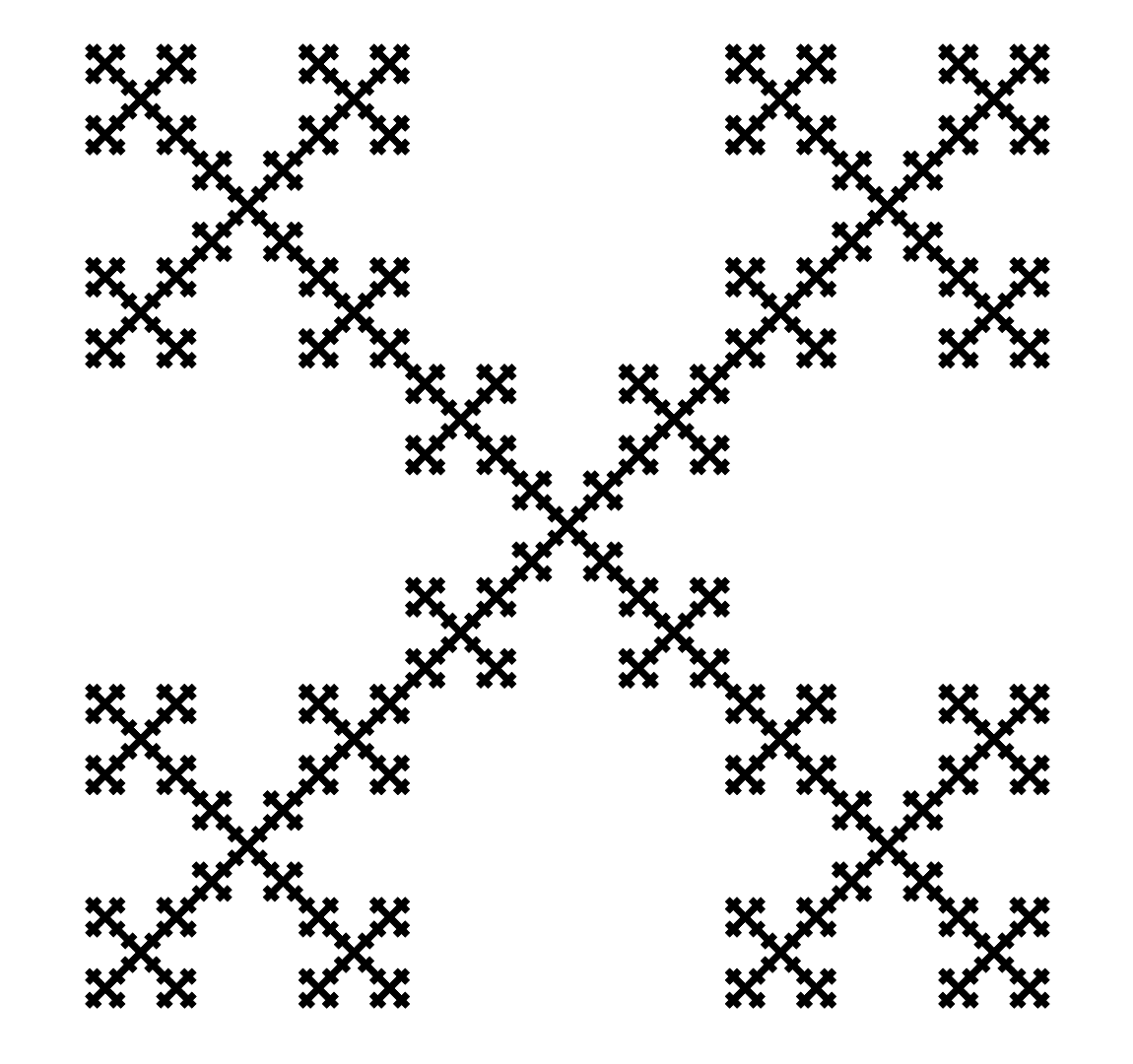}
%\\\ \\\includegraphics[trim={63 0 93 63},clip,height=0.20\textwidth]{SG}
%\\\ \\\includegraphics[trim={8 23 38 70},clip,height=.20\textwidth] {sc1.pdf}
\caption{A part of an infinite, or unbounded, Vicsek set.} 
\label{fig-Vicsek}
\end{figure}

In a recent work  \cite{devyver}, heat kernel gradient estimates have been obtained on some fractal-like cable systems for which the heat kernel has Gaussian estimates in small time and sub-Gaussian estimates in large time. In the present paper, we are interested in proving heat kernel gradient  estimates in a metric measure space $(X,d,\mu)$, the unbounded Vicsek set (Figure \ref{fig-Vicsek}), for which the heat kernel has the following  sub-Gaussian estimates for every $t>0$:
\[
\frac{c_1}{t^{d_h/d_w}}\exp\left(-c_2\Bigl(\frac{d(x,y)^{d_{w}}}{t}\Bigr)^{\frac{1}{d_{w}-1}}\right) \le p_t(x,y) \le \frac{c_3}{t^{d_h/d_w}}\exp\left(-c_4\Bigl(\frac{d(x,y)^{d_{w}}}{t}\Bigr)^{\frac{1}{d_{w}-1}}\right).
\]
Here, the parameter $d_h=\frac{\log 5}{\log 3}$ is the Hausdorff dimension of the metric space $(X,d)$ with  the Hausdorff measure $\mu$, and $d_w=d_h+1>2$ is a parameter which is called the walk dimension. Our interest in the Vicsek set comes from the fact that it provides a simple example of a space with sub-Gaussian heat kernel estimates and which does not admit a carr\'e du champ operator. In such settings, heat kernel gradient estimates have never been obtained before and we stress that  for the Vicsek set such estimates do not easily follow from the estimates obtained in \cite{devyver} on the cable systems approximations.

A first key point is to understand what a gradient estimate means in that setting. Despite the lack of a carr\'e du champ operator, we recently used in \cite{BC} a notion of weak gradient $\partial$  similar to the one appearing in  \cite{AEW}. The key property that allows to define this gradient is that the Vicsek set is a metric tree. Intuitively,  the gradient of a function is defined on a dense but zero Hausdorff measure subset of the Vicsek set, the so-called skeleton $\mathcal S$, through the fundamental theorem of calculus:
\[
f(x)-f(0)=\int_\gamma \partial f,
\]
where $0$ is the root of the Vicsek set and $\gamma$ the unique geodesic between $0$ and $x$. The function $\partial f$ can then be seen as a function on $\mathcal S$ which is locally integrable with respect to the length measure $\nu$ on $\mathcal S$. It is worth noting, and at the heart of the many difficulties arising in our analysis, that due to the fractal nature of the Vicsek set, the measure $\nu$  is not a Radon measure and is singular with respect to the measure $\mu$. More precisely, $\nu$ is supported on one-dimensional sets while $\mu$ which is the Hausdorff measure is supported on $d_h$-dimensional sets. 

A first main result of the paper is  the following theorem, see Corollary \ref{gb heat 1}.

\begin{theorem}\label{GEVintro}
 In the unbounded Vicsek set, the heat kernel $p_t(x,y)$ satisfies the following gradient estimate: for every $x\in X$ and $\nu$ a.e. $y \in \mathcal S$
    \begin{align}\label{GBVicsek1}
 | \partial_y p_t(x,y)| \le \frac{C}{t^{1/d_w}}  p_{c t}(x,y), \quad t >0.
 \end{align}
\end{theorem}

The key idea and novelty  in this estimate  is to prove a substantial improvement of the weak Bakry-\'Emery  estimates that were first  obtained in \cite{ABCRST3}, see Theorem \ref{Lipschitz pt}. After proving the pointwise gradient estimate \eqref{GBVicsek1}, we will shift our focus to the study of $L^p$ gradient bounds and we obtain the following, see Lemma \ref{lemma int}:

\begin{theorem}\label{GEVp intro}
 For $p \ge 1$, the heat kernel $p_t(x,y)$ satisfies the following $L^p$ gradient estimate: for every $x \in X$
    \begin{align}\label{GBpVicsek1}
 \int_\mathcal{S} | \partial_y p_t(x,y)|^p d\nu(y) \le \frac{C}{t^{p-1/d_w}}   \quad t >0.
 \end{align}
\end{theorem}

It should be noted that \eqref{GBpVicsek1} is by no means simply obtained by integrating \eqref{GBVicsek1} since it turns out that the function $y\to p_{ t}(x,y)$ is not even in $L^p (\mathcal S,\nu)$ for $p <+\infty$. To prove \eqref{GBpVicsek1}, we will instead use  the adjoint operator of $\partial$ and a Poincar\'e-type inequality it satisfies (see Lemma \ref{poinc_div}).

As an application of the heat kernel gradient estimates we prove the following family of Nash inequalities
\[
\| f \|_{L^p(X,\mu)} \le C \| \partial f \|_{L^p(\mathcal S,\nu)}^{\theta} \| f \|^{1-\theta}_{L^1(X,\mu)}, \quad p>1,
\]
where $\theta=\frac{(p-1)(d_w-1)}{pd_w-1}$ and also obtain some embedding theorems related to fractional Laplacians (see Theorems \ref{Nash Vicsek} and \ref{frac Riesz}). Those embeddings  allow us to raise a natural open question concerning Riesz transforms on the Vicsek set:

\begin{conjecture}
Do we have for  $p >1$,
\[
\| (-\Delta)^{\alpha_p} f \|_{L^p(X,\mu)}  \simeq \| \partial f \|_{L^p(\mathcal S,\nu)}
\]
with $\alpha_p=\left( 1-\frac{2}{p} \right) \frac{1}{d_w}+\frac{1}{p}$ ?
\end{conjecture}

In the second part of the paper we prove the existence of a semigroup $\vec{P_t}$ on $L^2(\mathcal S,\nu)$ that satisfies the intertwining
\[
\partial P_t=\vec{P_t} \partial,
\]
where $P_t$ is the heat semigroup on $L^2(X,\mu)$ with kernel $p_t(x,y)$. By  analogy with the situation on Riemannian manifolds,  it is natural to  call $\vec{P_t}$ the Hodge semigroup on the Vicsek set.  Our main result concerning this Hodge semigroup is the following, see Theorem \ref{bound hodge2}.

\begin{theorem}
The Hodge semigroup $\vec{P}_t$ admits a  kernel $\vec{p}_t (x,y)$ that satisfies the estimate: For $\nu \otimes \nu$ a.e. $x,y \in \mathcal S$,
    \[
 | \vec{p}_t (x,y)| \le  \frac{C  }{t^{1/d_w}}  \exp\biggl(-c\Bigl(\frac{d(x,y)^{d_{w}}}{t}\Bigr)^{\frac{1}{d_{w}-1}}\biggr), \quad t >0.
 \]
\end{theorem}
Finally in the last part of the paper we show how the previous results can  appropriately be modified  if the underlying space is the bounded Vicsek set.

%Let us  point out that for simplicity of the presentation, we restrict our analysis to the planar Vicsek set. However our definitions and results trivially extend to the higher dimensional Vicsek sets. More generally, it appears reasonable to infer that our approach is general and might be extended to a large class  of volume doubling metric trees which are increasing limits of cable systems trees. This will possibly be investigated in a later work.
%
    \paragraph{Notations:} Throughout the paper, we use the letters $c,C, c_1, c_2, c_3, c_4$  to denote positive constants which may vary from line to line.

\section{Preliminaries}

\subsection{Vicsek set}\label{notations Vicsek}

\begin{figure}[htb]\label{figure1}
 \noindent
 \makebox[\textwidth]{\includegraphics[height=0.22\textwidth]{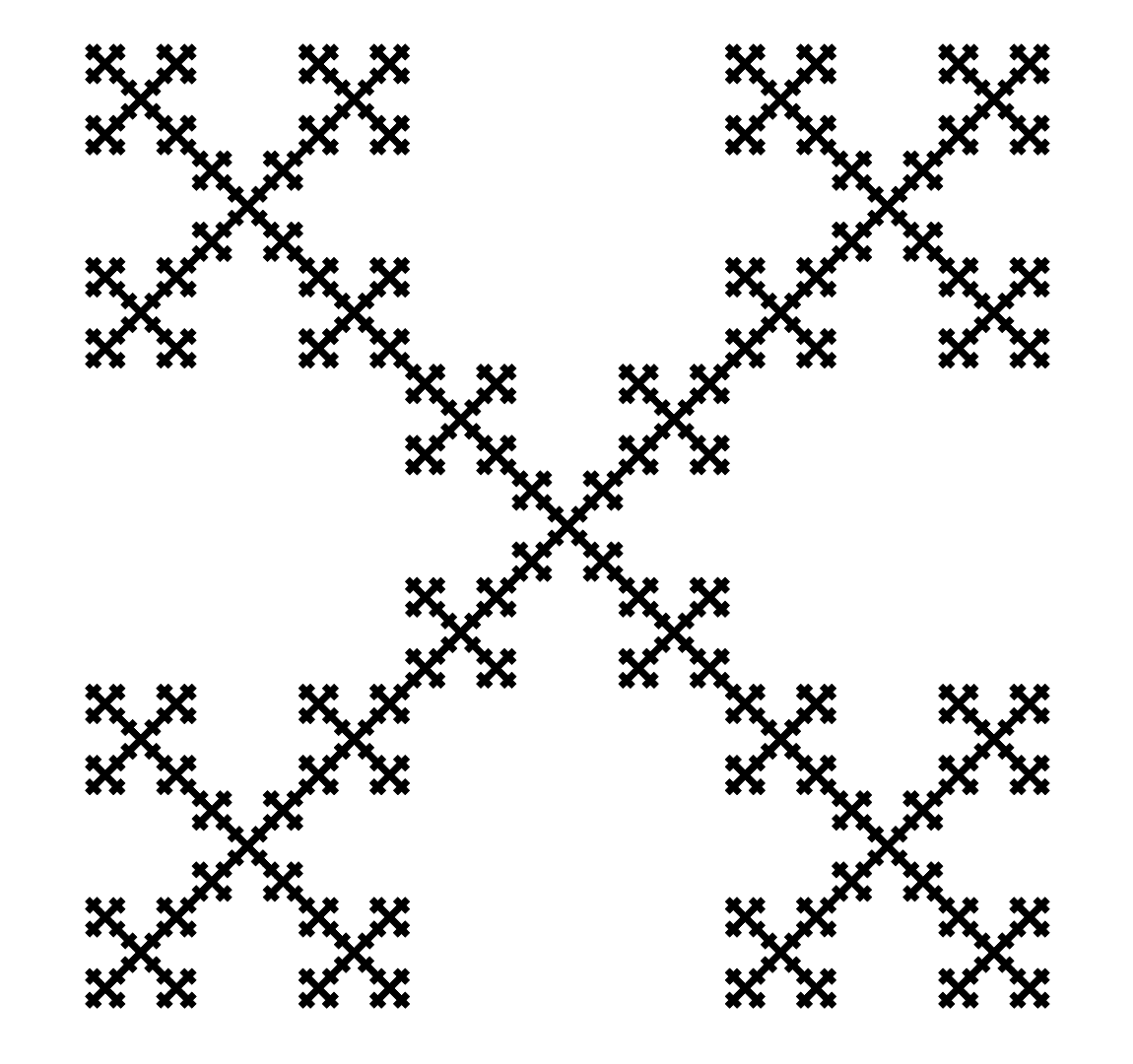}}
  	\caption{Bounded Vicsek set $K$}
\end{figure}

Let  $q_1=(0,0)$ be the center  of the unit square in the  plane  and let $q_2=(-\sqrt{2}/2,\sqrt{2}/2)$, $q_3=(\sqrt{2}/2,\sqrt{2}/2)$ , $q_4=(\sqrt{2}/2,-\sqrt{2}/2)$, and $q_5=(-\sqrt{2}/2,-\sqrt{2}/2)$ be the 4 corners of the square.
%Let $q_1=(-\sqrt{2}/2,\sqrt{2}/2)$, $q_2=(\sqrt{2}/2,\sqrt{2}/2)$ , $q_3=(\sqrt{2}/2,-\sqrt{2}/2)$, and $q_4=(-\sqrt{2}/2,-\sqrt{2}/2)$ be the 4 corners of the unit square in the plane and let $q_5=(0,0)$ be the center of that square. 
Define $\psi_i(z)=\frac13(z-q_i)+q_i$ for $1\le i\le 5$. The Vicsek set $K$ is the unique non-empty compact set such that 
\[
K=\bigcup_{i=1}^5 \psi_i(K)
\]
and the unbounded Vicsek set  $X$ is defined by
\[
X=\bigcup_{m=1}^{\infty} 3^m K .
\]
Let $W_0=\{q_1, q_2, q_3,q_4,q_5\}$. We define a  sequence of  sets of vertices  $\{W_n\}_{n\ge 0}$ inductively by
\[
W_{n+1}=\bigcup_{i=1}^5 \psi_i(W_n) \subset K.
\]
%and then consider 
%\[
%V_n=\bigcup_{m=1}^{\infty} 3^m W_n \subset X.
%\]
By definition, a cable system with vertices in $W_n$ is a union of segments, called edges, whose extremities are in $W_n$.  The unique connected cable system with vertices in $W_n$ and included in $K$ will be denoted $\bar{W}_n$. 
 \begin{figure}[htb]\label{figure2}
 \noindent
 \makebox[\textwidth]{\includegraphics[height=0.20\textwidth]{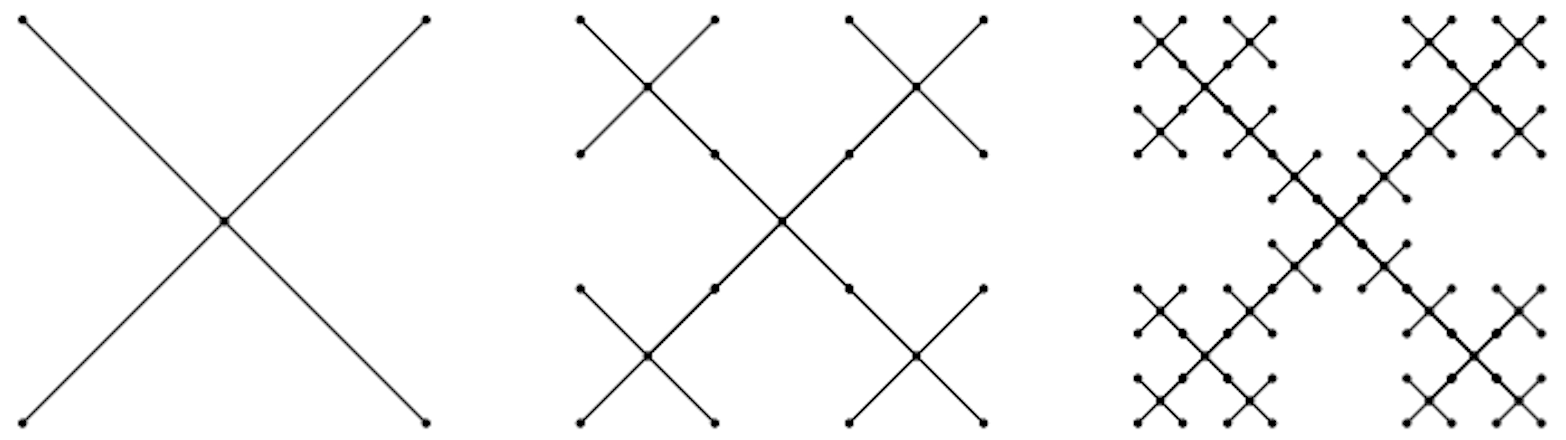}}
  	\caption{Cable systems $\bar{W}_0$, $\bar{W}_1$ and $\bar{W}_2$}
\end{figure}
%
%Denote $W=\{1, 2,3,4, 5\}$ and $W_n=\{1, 2,3,4, 5\}^n$ for $n\ge 1$. For any $w=\{i_1, \cdots, i_n \}\in W_n$, we denote by $\Psi_w$ the contraction mapping $\psi_{i_1}\circ \cdots\circ \psi_{i_n}$ and write $K_w:=\Psi_w(K)$. The set $K_w$ is called  an $n$-simplex. Let $V_0=\{q_1, q_2, q_3,q_4,q_5\}$. 
%We define a sequence of  sets of vertices  $\{V_n\}_{n\ge 0}$ inductively by
%\[
%V_{n+1}=\bigcup_{i=1}^5 \psi_i(V_n).
%\]
%For any $w=\{i_1, \cdots, i_n \}\in W_n$, we will denote $V_n^w=\Psi_w(V_n)$.
%Let then $\bar{V}_0$ be the cable system with vertices $V_0=\{q_1,q_2, q_3, q_4,q_5\}$
% and  consider the sequence of cable systems  $\bar{V}_n$ with vertices in $V_n$ inductively defined as follows.  The first cable system is  $\bar{V}_0$ and then 
%\[
%\bar{V}_{n+1} = \bigcup_{i=1}^5 \psi_i (\bar{V}_n).
%\]
%Note that $\bar{V}_n \subset K$ and that $K$ is the closure of $\cup_{n \ge 0} \bar{V}_n $. 
We will denote
\[
 V_n=\bigcup_{m \ge 1} 3^{m-n} W_m, \quad \bar{V}_n=\bigcup_{m \ge 1} 3^{m-n} \bar{W}_m.
\]
The set 
\[
\mathcal{S}= \bigcup_{n \ge 0} \bar{V}_n
\]
is called the skeleton of $X$, it is a dense subset in $X$.  We have then a natural sequence of  cable systems $\{\bar{V}_n\}_{n\ge 0}$ whose edges have length $3^{-n}$ and whose set of vertices is $V_n$. 

If $u,v$ are adjacent vertices in  $\bar{V}_n$ connected by an edge, we will write $u \sim v$ and say that $u \le v$ if the geodesic distance from the center 0 (i.e., $q_1$)  to $u$ in  $\bar{V}_n$ is less than the geodesic distance from $0$ to $v$. We will denote by $\mathbf{e}(u,v)$  the unique  edge in $\bar{V}_n$ connecting $u$ to $v$. 
 
We note that  $\bar{V}_n$ is a metric tree when equipped with the geodesic distance. Given any $u,v \in \bar{V}_n$, we will denote by $\gamma (u,v)$ the unique geodesic in $\bar{V}_n$ connecting $u$ to $v$.

\subsection{Geodesic distance and measures}

On $X$, we will consider the  geodesic distance $d$. More precisely, for $x,y \in X$, the distance $d(x,y)$ is defined as the infimum of the length of the rectifiable curves $\gamma:[0,1] \to \mathbb{R}^2$ such that $\gamma(0)=x$, $\gamma(1)=y$ and $\gamma(s) \in X$ for every $s \in [0,1]$. Note that the metric space $(X,d)$ is a metric tree.

The Hausdorff  measure $\mu$ of $(X,d)$ is the unique measure on $X$ such that for every $i_1, \cdots, i_n \in \{1,2,3, 4,5\}$ and $m \ge 0$
\[
\mu(3^m \psi_{i_1} \circ \cdots \circ \psi_{i_n} (K))=5^{m-n}.
\]
The Hausdorff dimension of $X$ is then $d_h=\frac{\log 5}{\log 3}$ and the metric space $(X,d)$ is $d_h$-Ahlfors regular in the sense that there exist constants $c,C>0$ such that for every $x \in K$, $r \ge 0$,
\[
c \,r^{d_h} \le \mu (B(x,r)) \le Cr^{d_h},
\]
where $B(x,r)=\left\{ y \in X, d(x,y) \le r \right\}$ denotes the closed ball with center $x$ and radius $r$.

There is also a reference measure $\nu$ on the skeleton $\mathcal{S}$, the length measure. It is characterized by the property that for every edge $\mathbf{e}$ in $\bar{V}_n$ connecting two neighboring vertices:
\[
\nu(\mathbf{e})= 3^{-n}.
\]
The measure $\nu$ is neither finite nor a Radon measure, since the skeleton obviously has infinite length and the measure of any  ball with positive radius is infinite. It is also clear that $\nu$ is singular with  respect to the Hausdorff measure $\mu$, since  the  skeleton has $\mu$-measure zero. However, $\nu$ is $\sigma$-finite on the $\sigma$-field generated by the $\mathbf{e}(x,y)$, $x,y \in V_n$, $x \sim y$, $n \ge 0$. We note that, up to a non-negative constant,  $\nu$ is the unique $\sigma$-finite measure on $\mathcal{S}$ such that for every subset $A \subset \mathcal{S}$ with $\nu(A)<+\infty$  and every affine map $\Phi :\mathbb{C} \to \mathbb{C}$, $\Phi(z)=\alpha z+\beta$ such that $\Phi (A) \subset \mathcal S$, one has  $\nu(\Phi (A)) =| \alpha | \nu(A) $.

%The measure $\nu$ is not finite (because the skeleton obviously has infinite length) but it is $\sigma$-finite on the $\sigma$-field generated by the $\mathbf{e}(x,y)$, $x,y \in V_n$, $x \sim y$, $n \ge 0$. The measure  $\nu$ is not a Radon measure neither since the measure of any  ball with positive radius is infinite. From the definition, it is also clear that $\nu$ is singular with  respect to the Hausdorff measure $\mu$ since  the  skeleton has $\mu$-measure zero. 

\subsection{Weak gradients, Sobolev spaces}

We denote by $L^1_{\mathrm loc} (\mathcal{S} ,\nu)$ the space of  measurable functions with respect to the $\sigma$-field generated by  the $\mathbf{e}(x,y)$, $x,y \in V_n$, $x \sim y$, $n \ge 0$ and $\nu$-integrable on each of those $\mathbf{e}(x,y)'s$.

\begin{definition}
Let $f \in C(X)$. We say that $f$ admits a weak gradient if there exists a measurable function $\partial f \in L^1_{\mathrm loc} (\mathcal{S} ,\nu)$ such that that for every $n$ and for every adjacent $u,v\in V_n$ with $u\le v$
\begin{equation}\label{eq:weakGradient}
f(v)-f(u)=\int_{\mathbf e(u,v)} \partial f d\nu.
\end{equation}
\end{definition}

The set $\bar{V}_n$ is a cable system.  As such, see for instance Section 5.1 in  \cite{MR3896108}, one can see any continuous function on $\bar{V}_n$ as a collection of functions $(f_\mathbf{e} )_{\mathbf{e} \in E_n}$, where $E_n$ is the set of edges of $\bar{V}_n$ and $f_\mathbf{e}: \left[0, 3^{-n} \right] \to \mathbb{R}$ is a continuous function with the appropriate boundary conditions.  For $f \in C(X)$ let us  denote $f^n=f|_{\bar{V}_n}$. One can then see that $f$ admits a weak gradient if and only if we have for all $\mathbf{e}$ in $E_n$, $f^n_\mathbf{e} \in W^{1,1}\left(\left[0, 3^{-n} \right]\right)$, where for an interval $I \subset \mathbb{R}$, $W^{1,1}(I)$ is the  usual $(1,1)$ Sobolev space for the Lebesgue measure. The following proposition is immediate from properties of  weak derivatives on the real line.

\begin{proposition}[Chain, Leibniz, and scaling rules]\label{Chain rule}
\

\begin{enumerate}
 \item Let $f \in C(X)$ and let $\Phi$ be a $C^1$ function on $\mathbb R$. If $f$  admits a weak gradient, then $\Phi(f)$ admits a weak gradient and
 \[
 \partial \Phi(f)=\Phi'(f)\partial f.
 \]
 \item Let $f,g \in C(X)$. If $f$ and $g$ admit weak gradients, then $fg$ admits a weak gradient and
 \[
 \partial (fg)=f \partial g +g \partial f.
 \]
 \item Let $f\in C(X)$. If $f$ admits a weak gradient, then for every $i_1, \cdots, i_n \in \{1,2,3, 4,5\}$, the function $f\circ \psi_{i_1} \circ \cdots \circ \psi_{i_n}$ admits a weak gradient and 
 \[
\partial (f \circ\psi_{i_1} \circ \cdots \circ \psi_{i_n})= 3^{-n} (\partial f  ) \circ \psi_{i_1} \circ \cdots \circ \psi_{i_n}.
\]
\end{enumerate}
\end{proposition}

\begin{definition}
Let $1 \le p\le \infty$. For $f\in C(X)$, we say that $f\in D^{1,p}(X)$ if it admits a weak gradient $\partial f\in L^p(\mathcal S, \nu)$. We will denote $W^{1,p}(X)=L^p(X,\mu) \cap D^{1,p}(X)$.
\end{definition}

The seminorm on $D^{1,p}(X)$ is defined by 
\[
\|f\|_{D^{1,p}(X)}:=\|\partial f\|_{L^p(X,\nu)}.
\]
One can see that $D^{1,\infty}(X)$ is the space of Lipschitz continuous functions on $X$ and that $\|\partial f\|_{L^\infty(X,\nu)}$ is the Lipschitz constant $\mathrm{Lip} (f)$ of $f$.
We refer to \cite[Section 3]{BC} for equivalent characterizations in the compact Vicsek set of the space $W^{1,p}(X)$, in terms of Korevaar-Schoen-Sobolev seminorms and  discrete $p$-energies.   

We note that it easily follows from the definition that if $f\in D^{1,p}(X)$, then for every $x,y \in \mathcal{S}$
\[
| f(x) -f(y) | \le \int_{\gamma (x,y)} | \partial f | d\nu,
\]
where we recall that $\gamma (x,y)$ denotes the unique geodesic in $\mathcal S$ connecting $x$ to $y$. From H\"older's inequality, we get that for every $x,y \in \mathcal S$
\begin{equation}\label{eq:Holder}
| f(x) -f(y) | \le d(x,y)^{1-\frac{1}{p}} \left(\int_{\gamma (x,y)} | \partial f |^p d\nu\right)^{1/p} 
\le d(x,y)^{1-\frac{1}{p}} \|\partial f\|_{L^p(X,\nu)}.    
\end{equation}
This inequality then holds for every $x,y \in X$ since $f$ is continuous and $\mathcal S$ is dense in $X$. 
Therefore for $p>1$, any $f \in D^{1,p}(X)$ is $1-\frac{1}{p}$ H\"older continuous. 

We now collect some basic properties of the weak gradient that will be important in the sequel.

\begin{proposition}\label{proposition closed}
Let $1 < p \le +\infty$. The operator $\partial: W^{1,p}(X) \to L^p(\mathcal S,\nu)$ is closed.
\end{proposition}

\begin{proof}
Let $f_n$ be a sequence in $W^{1,p}(X)$ that converges to some $f$ in $L^p(X,\mu)$ and such that $\partial f_n$ converges  in $L^p(\mathcal S,\nu)$ to some $g \in L^p(\mathcal S,\nu)$. It follows from \eqref{eq:Holder} that for every $n$ and for every  $u,v\in \mathcal S$
\begin{equation*}
|f_n(v)-f_n(u)| \le d(u,v)^{1-\frac{1}{p}} \|\partial f_n\|_{L^p(X,\nu)}.
\end{equation*}
 Since $\partial f_n$ converges  in $L^p(\mathcal S,\nu)$, there exists therefore a constant $C>0$ such that  for every $n$ and for every  $u,v\in \mathcal S$
\begin{equation*}
|f_n(v)-f_n(u)| \le C d(u,v)^{1-\frac{1}{p}}.
\end{equation*}
By continuity of $f_n$ and density of $\mathcal S$ in $X$, the above inequality holds true for every $u,v \in X$. Using the assumed convergence of $f_n$ to $f$ in $L^p(X,\mu)$ and the Arzela-Ascoli theorem we deduce that given $u,v \in  V_m$ with $u,v$ adjacent and $u \le v$, there exists a subsequence $f_{n_k}$ such that $f_{n_k}$ converges uniformly to $f$ on $\mathbf{e}(u,v)$. Note that
\[
f_{n_k}(v)-f_{n_k}(u)=\int_{\mathbf{e}(u,v)} \partial f_{n_k} d\nu,
\]
one obtains by taking the limit that
\[
f(v)-f(u)=\int_{\mathbf{e}(u,v)} g d\nu.
\]
Since $g \in L^p(\mathcal S,\nu)$, we conclude that $f \in W^{1,p}(X)$ and  $g =\partial f$. %The proof is complete.
\end{proof}

\begin{proposition}
Let $1 < p \le +\infty$. The operator $\partial: D^{1,p}(X) \to L^p(\mathcal S, \nu)$ is surjective.
\end{proposition}

\begin{proof}
Let $\omega \in L^p(\mathcal S,\nu)$. For $v \in \mathcal S$, write
\[
f(v)=\int_{\gamma(0,v)} \omega (x) d\nu(x),
\]
where we recall that $\gamma(0,v)$ denotes the unique geodesic path from $0$ to $v$. As before, one can see that for every $u,v \in \mathcal S$,
\[
| f(v)-f(u)| \le d(u,v)^{1-\frac1p} \left( \int_{\mathcal S} | \omega |^p  d\nu  \right)^{1/p}.
\]
Therefore $f$ admits a unique H\"older continuous extension to $X$ that we still denote by $f$. We have then $f \in D^{1,p}(X)$ and $\partial f=\omega$.
\end{proof}

The following notion of piecewise affine function will play an important role.

\begin{definition}
A  continuous function $\Phi:X \to \mathbb{R}$ is called $n$-piecewise affine, if there exists $n \ge 0$ such that $\Phi$ is piecewise affine on the cable system $\bar{V}_n$ (i.e., linear between the vertices of $\bar{V}_n$)  and constant on any connected component of $\bar{V}_m \setminus \bar{V}_n$ for every $m >n$.  In other words, a continuous function $\Phi:X \to \mathbb{R}$ is  $n$-piecewise affine if
\[
\Phi(x)=\Phi(0) + \int_{\gamma(0,x) }\eta d\nu, \quad x\in \mathcal{S},
\]
where $\eta$ is a function supported in $\bar{V}_n$ and takes constant value on each of the edges of $\bar{V}_n$.
 \end{definition}

We have the following approximation result.

\begin{lemma}\label{approxi L4}
Let $1 \le p < +\infty$. For any $\eta \in L^p(\mathcal S,\nu) \cap L^\infty (\mathcal S,\nu)$, there exists a sequence $\phi_n$ such that:
 \begin{enumerate}
 \item $ \phi_n \in  W^{1,q}(X)$, $p \le  q \le +\infty$;
     \item $\partial \phi_n \to \eta$ in $L^q (\mathcal S,\nu)$, $p < q <+\infty$;
     \item $\| \partial \phi_n \|_{L^q(\mathcal S,\nu)}\le 2 \| \eta \|_{L^\infty(\mathcal S,\nu)} $, $p \le q \le +\infty$.
 \end{enumerate}
\end{lemma}

\begin{proof}
We first construct a convenient family of cutoff functions. Consider the unique 1-piecewise affine function $h$ on $X$ which is equal to 1 on the vertices $\frac{1}{3} W_0$ and zero on all of the other vertices of $V_1$. Let
\[
h_n(x)=h ( x/3^n), \, x \in X.
\]
We have then for every $x \in X$,  $ h_n(x) \to 1$ when $n \to \infty$ and moreover 
\begin{equation}\label{eq:partial-hn}
| \partial h_n (x)| \le 3^{-n} 1_{3^n\bar{W}_0\setminus 3^{n-1}\bar{W}_0 }(x).
\end{equation}
For $\eta \in L^p(\mathcal S,\nu)\cap L^\infty (\mathcal S,\nu)$, we define
\[
\phi_n(x) =h_n (x)\int_{\gamma (0,x)} \eta d\nu, \quad x \in \mathcal S.
\]
By continuity, $\phi_n$ admits an extension to $X$, which is still denoted by $\phi_n$. Since $\phi_n$ is continuous and compactly supported, we have $\phi_n \in L^q(X,\mu)$ for $p \le q \le +\infty$. 

Next we observe that for $\nu$ a.e. $x \in \mathcal S$
\begin{align}\label{eq:partial phi}
\partial \phi_n(x) =\partial h_n (x)\int_{\gamma (0,x)} \eta d\nu+ h_n(x) \eta (x).
\end{align}
Since $\eta \in L^q(\mathcal S,\nu)$ for $q \ge p$, it is easy to see that $h_n \eta \to \eta$ in $L^q(\mathcal{S},\nu)$ when $q \ge p$. On the other hand,  we have from \eqref{eq:partial-hn} and H\"older's inequality that
\begin{align*}
\int_\mathcal{S} \left|\partial h_n (x)\int_{\gamma (0,x)} \eta d\nu \right|^q d\nu(x)  
&\le \int_\mathcal{S} \left(3^{-n} 1_{3^n\bar{W}_0\setminus 3^{n-1}\bar{W}_0 }(x)\right)^q \left(\|\eta \|_{L^p(\mathcal S,\nu)}  d(0,x)^{1-\frac{1}{p}}\right)^qd\nu(x)  
\\ &\le 3^{-nq} \|\eta \|_{L^p(\mathcal S,\nu)}^q \int_\mathcal{S} 1_{3^n\bar{W}_0\setminus 3^{n-1}\bar{W}_0 }(x)d(0,x)^{q-\frac{q}{p}}d\nu(x)
\\ & \le 3^{-nq} \|\eta \|_{L^p(\mathcal S,\nu)}^q 3^{n\left( q-\frac{q}{p}\right)} \nu(3^n\bar{W}_0\setminus 3^{n-1}\bar{W}_0)
\\ & \le 3^{n\left( 1-\frac{q}{p}\right)} \|\eta \|_{L^p(\mathcal S,\nu)}^q.
\end{align*}
\begin{comment}
\[
\left|\partial h_n (x)\int_{\gamma (0,x)} \eta d\nu \right|\le  \|\eta \|_{L^p(\mathcal S,\nu)} 3^{-n} 1_{3^n\bar{W}_0\setminus 3^{n-1}\bar{W}_0 }(x) d(0,x)^{1-\frac{1}{p}}
\]
and 
\[
\int_{\mathcal S} 1_{3^n\bar{W}_0\setminus 3^{n-1}\bar{W}_0 }(x) d(0,x)^{q-\frac{q}{p}} d\nu(x) \le  3^{n\left( q-\frac{q}{p}\right)} \nu(3^n\bar{W}_0\setminus 3^{n-1}\bar{W}_0) \le C 3^{n\left( 1+q-\frac{q}{p}\right)}.
\]
Hence
\[
\int_\mathcal{S} \left|\partial h_n (x)\int_{\gamma (0,x)} \eta d\nu \right|^q d\nu(x) \le 
3^{n\left( 1-\frac{q}{p}\right)} \|\eta \|_{L^p(\mathcal S,\nu)}^q,
\]
\end{comment}
Hence the integral is bounded for $p=q$, and converges to zero when $n\to +\infty$ for $p<q<\infty$. 
\begin{comment}
Hence for $q>p$ when $n\to +\infty$
\[
\int_\mathcal{S} \left|\partial h_n (x)\int_{\gamma (0,x)} \eta d\nu \right|^q d\nu(x) \to 0,
\]
while for $p=q$ we get that $\int_\mathcal{S} \left|\partial h_n (x)\int_{\gamma (0,x)} \eta d\nu \right|^q d\nu(x)$ is bounded.
\end{comment}
We deduce that $\partial \phi_n \to \eta$ in $L^q (\mathcal S,\nu)$ for $p < q <+\infty$, and $\partial \phi_n$ is bounded in $L^p (\mathcal S,\nu)$. %So claims $(\mathrm{i})$, $(\mathrm{ii})$ and $(\mathrm{iii})$ hold.
Finally, we have $\nu$ a.e. $x\in \mathcal S$
\begin{align*}
|\partial \phi_n(x) | & \le |\partial h_n (x) | \, \left|\int_{\gamma (0,x)} \eta d\nu \right|+ h_n(x) |\eta (x) | \\
 & \le 3^{-n} 1_{3^n\bar{W}_0\setminus 3^{n-1}\bar{W}_0 } (x)d(0,x) \| \eta \|_{L^\infty(\mathcal S,\nu)} +\| \eta \|_{L^\infty(\mathcal S,\nu)} \\
 & \le 2 \| \eta \|_{L^\infty(\mathcal S,\nu)}.
\end{align*}
%This concludes $(\mathrm{iv})$ and we finish the proof.
\end{proof}

\begin{proposition}\label{dense range}
Let $1 < p < +\infty$. The range of the operator $\partial: W^{1,p}(X) \to L^p(\mathcal S, \nu)$ is dense in $L^p(\mathcal S, \nu)$.
\end{proposition}

\begin{proof}
Let  $1 < p < +\infty$. Since $L^1(\mathcal S,\nu)\cap L^\infty(\mathcal S,\nu)$ is dense in $L^p(\mathcal S,\nu)$, it is enough to prove that the range of $\partial$ is $L^p$ dense in $L^1(\mathcal S,\nu)\cap L^\infty(\mathcal S,\nu)$, which immediately follows from Lemma \ref{approxi L4} $(\mathrm{ii})$.
\end{proof}

\section{Heat kernel gradient  bounds}

\subsection{Dirichlet form, Heat kernel}

We now introduce the canonical Dirichlet form on $X$ and recall some of the basic properties of its associated heat kernel with respect to the measure $\mu$. For $f,g \in W^{1,2}(X)$, we define
\[
\mathcal{E}(f,g)=\int_{\mathcal S} \partial f \,  \partial g \, d\nu.
\]
We note that if $\Phi:X \to \mathbb{R}$ is  $n$-piecewise affine, then for every $m>n$
\[
\mathcal{E}(\Phi,\Phi)=\frac{1}{2} 3^{m}  \sum_{x,y \in V_m, x \sim y} |\Phi(x)-\Phi(y)|^2.
\]
Using then known result about Dirichlet forms as limits of self-similar energies on infinite nested fractals as in \cite[Theorem 7.14]{Barlow} and \cite[Theorem 2.7]{FHK} (see also Theorem 2.9 in \cite{BC}), we have the following result.
%and using similar arguments as in  we have:

\begin{theorem}
The quadratic form $\mathcal{E}$ with domain $W^{1,2}(X)$ is a strongly local and regular Dirichlet form in $L^2(X,\mu)$. A core for $\mathcal{E}$ is the set of compactly supported piecewise affine functions. Moreover, for every $f \in W^{1,2}(X)$
\[
\mathcal{E}(f,f)=\lim_{m \to +\infty} \frac{1}{2} 3^{m}  \sum_{x,y \in V_m, x \sim y} |f(x)-f(y)|^2.
\]
\end{theorem}
%
%Let $f\in C(V^{\langle \infty\rangle})=\{f: V^{\langle \infty\rangle}\to \mathbb R\}$ and define 
%\[
%\mathcal E_n (f,f)=\frac12\rho^n\sum_{w\in W_n} (f\circ \psi_w(x)- f\circ \psi_w(y) )^2, 
%\]
%where   $\rho>1$ is the resistance scale factor of $K$. 
%The Dirichlet form on $K$, denoted by $\mathcal E$, is given by
%\[
%\mathcal E  (f,f)=\lim_{n\to \infty} \mathcal E_n (f,f), 
%\]
%where $f\in \mathcal F_K=\{f\in C(V^{\langle \infty\rangle}): \sup \mathcal E_n(f,f)<\infty\}$.
%For more details, we refer to, for instance, \cite[Section 2]{FHK} and \cite[Corollary 6.28, Section 7]{Barlow}). Then $(\mathcal E,\mathcal F_K)$ is a local regular Dirichlet form on $L^2(K, \mu)$. 

Let $\Delta $ be the generator of $(\mathcal E,W^{1,2}(X))$  in $L^2(X, \mu)$, i.e., for $f\in \mathrm{dom}(\Delta)$ and $g \in W^{1,2}(X)$,
\[
\int_\mathcal{S} \partial f \partial g d\nu=-\int_X g \Delta f d\mu.
\]
The associated heat semigroup $(P_t)_{t\ge 0}=(e^{t\Delta})_{t\ge 0}$ admits a continuous heat kernel with respect to the Hausdorff measure $\mu$ that we denote by $p_t(x,y)$. The  heat kernel $p_t(x,y)$  satisfies sub-Gaussian estimates, see \cite[Theorem 8.18]{Barlow} and \cite[Theorem 1]{FHK}. More precisely, for every $(x,y)\in X\times X$ and $t >0$
\begin{equation}\label{eq:sub-Gaussian1}
c_1t^{-\frac{d_h}{d_w}}\exp\biggl(-c_2\Bigl(\frac{d(x,y)^{d_{w}}}{t}\Bigr)^{\frac{1}{d_{w}-1}}\biggr) \le p_t(x,y) 
\le  c_3 t^{-\frac{d_h}{d_w}}\exp\biggl(-c_4\Bigl(\frac{d(x,y)^{d_{w}}}{t}\Bigr)^{\frac{1}{d_{w}-1}}\biggr),
\end{equation}
where $d_w:=d_h+1$ is often called the walk dimension.
Those sub-Gaussian estimates easily imply (see Lemma 2.3 in \cite{ABCRST3}) that  for every $\kappa \ge 0$, there are constants $c,C>0$ so that for every $x,y \in X$, and $t >0$
\begin{equation}\label{eq:boundpolyptbypct}
 d(x,y)^\kappa p_t(x,y)\leq C t^{\frac{\kappa}{d_w}} p_{ct}(x,y).
\end{equation}
%
%Furthermore,  the weak Bakry-\'Emery  non-negative curvature condition holds true on $K$  (see \cite[Theorem 3.7]{ABCRST3}), that is, there exists a constant $C>0$ such that for every $g \in L^\infty(X,\mu)$ and every $t>0$
%\begin{equation}\label{eq:weakBE1}
%|P_tg(x)-P_tg(y)|\le C\frac{d(x,y)^{d_w-d_h}}{t^{1-d_h/d_w}} \|g\|_{L^\infty(K,\mu)},
%\end{equation}

For later use we record the following bound which classically follows from the heat kernel estimates, see for instance \cite[Corollary 5]{MR1423289}.

\begin{lemma}\label{sub_time}
There exist constants $c,C>0$ such that for every $t >0$ and $x,y \in X$,
\begin{equation}\label{eq:sub-Gaussian2}
|\partial_t p_t(x,y) |
\le C  t^{-1-\frac{d_h}{d_w}}\exp\biggl(-c\Bigl(\frac{d(x,y)^{d_{w}}}{t}\Bigr)^{\frac{1}{d_{w}-1}}\biggr).
\end{equation}
\end{lemma}

\subsection{Heat kernel  Lipschitz estimates and gradient bounds}

We start with the following Lipschitz estimate for the heat kernel.

 \begin{theorem}\label{Lipschitz pt}
   There exist constants $c,C>0$ such that for every $x,y,z \in X$, and $t>0$
 \begin{align}\label{eq:pkl}
 | p_t(z,x)-p_t(z,y)| \le  C \frac{d(x,y)}{t^{1/d_w}}  ( p_{ct}(z,x)+p_{ct}(z,y)) .
 \end{align}
 \end{theorem}
 
 \begin{proof}

 Denote by 
 
\[ 
g_\lambda (x,y)=\int_0^{+\infty} e^{-\lambda t} p_t (x,y) dt, \quad \lambda >0,
\]
the resolvent kernel of $\Delta$, i.e.,
 \[
 G_\lambda f (x)=(\Delta-\lambda)^{-1} f (x)=\int_X g_{\lambda} (x,y) f(y) d\mu(y), \quad f \in L^2(X,\mu).
 \]
From the proof of \cite[Theorem 3.40]{Barlow} (see also \cite[Theorem 5.20]{BP}), one has the following estimate
\[
| g_{\lambda} (z,x)-g_{\lambda} (z,y)| \le C d(x,y) (p_{\lambda} (z,x)+p_{\lambda} (z,y)),
\]
where $p_{\lambda} (x,y):=\frac{g_\lambda(x,y)}{g_\lambda(y,y)}$ and $g_{\lambda}(y,y)\simeq\lambda^{\frac{d_h}{d_w}-1}$, thanks to \cite[Proposition 3.28 (a)]{Barlow}. We conclude that 
\[
| g_{\lambda} (z,x)-g_{\lambda} (z,y)| \le C\lambda^{1-\frac{d_h}{d_w}} d(x,y) (g_{\lambda} (z,x)+g_{\lambda} (z,y)).
\]
Integrating the estimate yields that for $f \in L^2(X,\mu)$
 \[
 | G_\lambda f (x)-G_\lambda f (y)|\le C \lambda^{1-\frac{d_h}{d_w}} d(x,y)(G_\lambda |f| (x)+G_\lambda |f| (y)).
 \]
 Thus, for any $f \in \mathrm{dom} (\Delta)$ and $\lambda >0$
 \[
 |  f (x)- f (y)|\le C \lambda^{1-\frac{d_h}{d_w}} d(x,y)(G_\lambda |(\Delta-\lambda) f| (x)+G_\lambda |(\Delta-\lambda) f| (y)).
 \]
 Now applying this inequality with $f=p_t(z,\cdot)$ gives
 \[
 |  p_t(z,x)- p_t(z,y)|\le C \lambda^{1-\frac{d_h}{d_w}} d(x,y)(G_\lambda |(\Delta-\lambda) p_t(z,\cdot)| (x)+G_\lambda |(\Delta-\lambda) p_t(z,\cdot)| (y)).
 \]
 We then have from Lemma \ref{sub_time}
 \begin{align*}
 G_\lambda |(\Delta-\lambda) p_t(z,\cdot)| (x)&=\int_X g_\lambda (x,u)|(\Delta-\lambda) p_t(z,u)| d\mu(u) \\
   &\le  C \left(\lambda+ \frac{1}{t} \right)\int_X g_\lambda (x,u)p_{c_1t}(z,u) d\mu(u).
 \end{align*}
By the semigroup property of the heat kernel, one has then
 \begin{align*}
 \int_X g_\lambda (x,u)p_{c_1t}(z,u) d\mu(u)&=\int_0^{+\infty}\int_X e^{-\lambda s} p_s(x,u) p_{c_1t}(z,u) d\mu(u) ds \\
  &=\int_0^{+\infty}e^{-\lambda s} p_{s+c_1t}(x,z) ds \\
%   &=e^{\lambda c t} \int_{ct}^{+\infty} e^{-\lambda u} p_u (x,z) du \\
   &\le e^{c_1\lambda  t} \int_{0}^{+\infty} e^{-\lambda u} p_u (x,z) du
\\ &    \le C \lambda^{\frac{d_h}{d_w}-1} e^{c_1
 \lambda t-c_2 \lambda^{1/d_w} d(x,z)},
 \end{align*}
 where the last inequality follows from the upper bound in \eqref{eq:sub-Gaussian1} by integrating with respect to $u$ (see \cite[Proposition 3.28]{Barlow}).
%One now uses the estimate
%\[
%\int_{0}^{+\infty} e^{-\lambda u} p_u (x,z) du \le C \lambda^{\frac{d_h}{d_w}-1}  e^{-c_2  \lambda^{1/d_w}d(x,z)}.
%\]
%which follows  from the upper bound in \eqref{eq:sub-Gaussian1} by integrating with respect to $u$ (see \cite[Proposition 3.28]{Barlow}).
 %Therefore, we have
 %\[
 %\int_X g_\lambda (x,u)p_{ct}(z,u) d\mu(u) \le C \lambda^{\frac{d_h}{d_w}-1} e^{c_1
 %\lambda t-c_2 \lambda^{1/d_w} d(x,z)},
 %\] 
We conclude
 \[
 | p_t(z,x)-p_t(z,y)|\le C d(x,y) \left(\lambda+ \frac{1}{t} \right)\left( e^{c_1
 \lambda t-c_2 \lambda^{1/d_w} d(x,z)} +e^{c_1
 \lambda t-c_2 \lambda^{1/d_w} d(y,z)} \right).
 \]
  
% Applying first this inequality with $\lambda=1/t$ gives
% \[
%  | p_t(z,x)-p_t(z,y)|\le c \frac{d(x,y)^{d_w-d_h}}{t}
% \]
% Then, we consider two cases
% 
% Case 1. $d(x,y) \ge \delta t^{1/d_w}$. In that case,
% \[
% | p_t(z,x)-p_t(z,y)|\le p_t(z,x)+p_t(z,y)
% \]
% and the inequality
% \[
% | p_t(z,x)-p_t(z,y)| \le c_1 \frac{d(x,y)^{d_w-d_h}}{t^{1-d_h/d_w}}  ( p_{c_2t}(z,x)+p_{c_2t}(z,y)) 
% \]
%follows from the heat kernel upper bound.
%
%Case 2. $d(x,y) \le \delta t^{1/d_w}$. 
%

Assume first $d(y,z) \ge d(x,z)$. One has then
\[
| p_t(z,x)-p_t(z,y)|\le C d(x,y) \left(\lambda+ \frac{1}{t} \right)e^{c_1
 \lambda t-c_2 \lambda^{1/d_w} d(x,z)}.
\]
If we choose $\lambda=  \alpha\left(\frac{d(x,z)}{t} \right)^{\frac{d_w}{d_w-1}}$ with a constant $\alpha>0$ small enough such that $c_2\alpha^{\frac1{d_w}}-c_1\alpha>0$ (such difference is denoted by $c_3$),
one obtains
\begin{align*}
| p_t(z,x)-p_t(z,y)| &\le C d(x,y) \left(\left(\frac{d(x,z)}{t} \right)^{\frac{d_w}{d_w-1}}+ \frac{1}{t} \right)\exp\left(-c_3\Bigl(\frac{d(x,z)^{d_w}}{t}\Bigr)^{\frac{1}{d_w-1}} \right)\\
  & \le \frac{C}{t} d(x,y)\exp\left(-c_4\Bigl(\frac{d(x,z)^{d_w}}{t}\Bigr)^{\frac{1}{d_w-1}} \right) \\
  & \le \frac{C}{t^{1-d_h/d_w}} d(x,y) p_{ct}(x,z),
\end{align*}
where we used the fact that
\[
d(x,z)^{\frac{d_w}{d_w-1}} \exp\left(-c_3\Bigl(\frac{d(x,z)^{d_w}}{t}\Bigr)^{\frac{1}{d_w-1}} \right)
\le C t^{\frac{1}{d_w-1}}\exp\left(-c_4\Bigl(\frac{d(x,z)^{d_w}}{t}\Bigr)^{\frac{1}{d_w-1}} \right).
\]
Similarly, if $d(y,z) \ge d(x,z)$, one obtains
\[
| p_t(z,x)-p_t(z,y)| \le \frac{C}{t^{1-d_h/d_w}} d(x,y) p_{ct}(y,z)
\]
and the conclusion follows.
 \end{proof}

 As a corollary, we  get an improved weak Bakry-\'Emery estimate (compared to \cite[Theorem 3.7]{ABCRST3}).
 
 \begin{corollary}\label{LipP_t}
 There exist constants $c,C>0$ such that for every $f \in L^p(X,\mu)+ L^q(X,\mu)$, $1 \le p,q \le +\infty$, $x,y \in X$ and $t>0$
 \[
 | P_t f (x)-P_tf(y)| \le C\frac{d(x,y)}{t^{1/d_w}}  ( P_{ct}|f|(x)+P_{ct}|f|(y)) .
 \]
 Therefore, for every $f \in  L^p(X,\mu)+ L^q(X,\mu)$, $1 \le p,q \le +\infty$ and $t>0$, $P_tf \in W^{1,\infty}(X)$. Moreover, there exist constants $c,C>0$ such that for every $f \in L^p(X,\mu)+ L^q(X,\mu)$, $1 \le p,q \le +\infty$, $\nu$ a.e. $x \in \mathcal S$ and every  $t>0$
 \begin{align}\label{GB-semigroup}
 | \partial P_t f (x)| \le    \frac{C}{t^{1/d_w}}  P_{ct}|f|(x) .
 \end{align}

 \end{corollary}
 
 \begin{proof}
 The first part of the corollary follows directly from Theorem \ref{Lipschitz pt}. Indeed, for every $f \in L^p(X,\mu)+ L^q(X,\mu)$, $x,y \in X$ and $t>0$
 \begin{align*}
 | P_t f (x)-P_tf(y)|&=\left| \int_X p_t(z,x) f(z) d\mu(z)- \int_X p_t(z,y) f(z) d\mu(z) \right| \\
  &\le \int_X | p_t(z,x)- p_t(z,y)|\,  | f(z) | d\mu(z) \\
  &\le  C   \frac{d(x,y)}{t^{1/d_w}}  \int_X ( p_{ct}(z,x)+p_{ct}(z,y)) | f(z) | d\mu(z) \\
  & \le C   \frac{d(x,y)}{t^{1/d_w}}  ( P_{ct}|f|(x)+P_{ct}|f|(y)).
 \end{align*}
Since both $P_t f$ and $P_{ct}|f|$ are bounded due to the upper bound in \eqref{eq:sub-Gaussian1}, we deduce that $P_tf \in W^{1,\infty}(X)$. Moreover, we note that for every $x,y \in \mathcal{S}$ such that $\gamma (x,y)$ does not intersect the center $0$ of $X$, one has
 \[
 | P_t f (x)-P_tf(y)| =\left| \int_{\gamma (x,y)} \partial P_tf (u) d\nu (u) \right|,
 \]
and therefore
 \[
 \left| \int_{\gamma (x,y)} \partial P_tf (u) d\nu (u) \right| \le  C  \frac{d(x,y)}{t^{1/d_w}}  ( P_{ct}|f|(x)+P_{ct}|f|(y)).
 \]
 Since $ P_{ct}|f|$ is continuous, it then easily follows from the Lebesgue differentiation theorem that  for $ \nu$ a.e. $x \in \mathcal S$ and every  $t>0$
 \[
 | \partial P_t f (x)| \le    \frac{C}{t^{1/d_w}}  P_{ct}|f|(x) .
 \]
 \end{proof}
 
 \begin{remark}\label{remark median}
 The inequality  \eqref{GB-semigroup} can be improved. Indeed, let $L\in \mathbb R$. Using  \eqref{GB-semigroup}  with the function $f-L$ yields that for $\nu$ a.e. $x \in \mathcal S$,
 \begin{align}\label{GB median}
  | \partial P_t f (x)| \le    \frac{C}{t^{1/d_w}}  P_{ct}(|f-L|)(x) 
  \end{align}
  However, repeating the proof of \eqref{GB-semigroup} shows that the set of $x \in \mathcal{S}$ for which \eqref{GB median} holds is independent from $L$ (it only depends on $f$). This gives therefore that for $\nu$ a.e. $x \in \mathcal S$,
 \begin{align*}
  | \partial P_t f (x)| \le    \frac{C}{t^{1/d_w}} \inf_{L \in \mathbb R}  P_{ct}(|f-L|)(x) .
  \end{align*}
 \end{remark}

We also obtain the following gradient estimate for the heat kernel.

\begin{corollary}\label{gb heat 1}
There exist $c,C>0$ such that for every $t >0$, $y \in X$, and $\nu$ a.e. $x \in \mathcal S$
\begin{align}\label{GB-heat kernel}
| \partial_x p_t(x,y)| \le  \frac{C}{t^{1/d_w}}  p_{ct}(x,y)
\end{align}
and such that for every $t >0$, $x \in X$, and $\nu$ a.e. $y \in \mathcal S$
\[
| \partial_y p_t(x,y)| \le  \frac{C}{t^{1/d_w}}  p_{ct}(x,y).
\]

\end{corollary}

\begin{proof}
This follows from Theorem \ref{Lipschitz pt} and the same arguments as in the proof of Corollary \ref{LipP_t}.
\end{proof}

\begin{remark}
  For a fixed $y \in X$,  the upper bound \eqref{GB-heat kernel} holds outside a set of $x$'s  of $\nu$-measure zero, but let us notice that this set a priori depends on $y$.
\end{remark}
\begin{remark}
  Using the sub-Gaussian upper bound \eqref{eq:sub-Gaussian1} one can rewrite the gradient estimate in Proposition \ref{gb heat 1} as
  \begin{align}\label{gb heat 2}
| \partial_x p_t(x,y)| \le  \frac{C}{t}  \exp\left(-c\Bigl(\frac{d(x,y)^{d_{w}}}{t}\Bigr)^{\frac{1}{d_{w}-1}}\right).
\end{align}
\end{remark}

\subsection{Continuity \texorpdfstring{$L^q \to W^{1,p}$}{LqW1p}  of  the heat semigroup }\label{sec:3.3}

In the previous section we obtained pointwise gradient estimates for the heat kernel. The goal of this section is to obtain $L^p$ gradient bounds. Let us observe that one can not obtain integrated gradient estimates  by simply integrating with respect to $\nu$ the upper bound \eqref{gb heat 2}. Indeed, since $\nu$ is not a Radon measure,  for $\nu$ a.e. $x \in \mathcal S$ the function $y \to \exp\left(-c\Bigl(\frac{d(x,y)^{d_{w}}}{t}\Bigr)^{\frac{1}{d_{w}-1}}\right)$ is not in $L^p(\mathcal S, \nu)$ for any $p \ge 1$. Instead we will obtain integrated gradient estimates using the co-differential operator and an associated Poincar\'e type inequality proved in the Proposition \ref{poinc_div} below.

The co-differential is defined as the  $L^2$ adjoint of the weak gradient $\partial$. This means that the operator $\partial^*$ is the operator  from $L^2(\mathcal{S},\nu)$ to $L^2(X,\mu)$ with domain
\[
\mathrm{dom} \, \partial^* := \left\{\eta \in L^2(\mathcal{S},\nu) :~\exists f\in L^2(X,\mu),~\text{with } \int_\mathcal{S}  \eta \partial \phi  d\nu = \int_X f \phi d\mu,  \forall~\phi\in W^{1,2}(X) \right\}.
\]
It follows from Theorem \ref{proposition closed} that the operator $\partial$ is closed and densely defined on $L^2(X,\mu)$. Therefore $\partial^*$ is a densely defined and closed operator, see Theorem VIII.1 in \cite{RS72}.  Given  $\eta \in \mathrm{dom} \, \partial^*$, there is a unique $f \in L^2(X,\mu)$ that satisfies
\[
\int_\mathcal{S}  \eta \partial \phi  d\nu = \int_X f \phi d\mu, \quad   \forall~\phi\in W^{1,2}(X)
\]
and we define then $\partial^* \eta = f$.  

\begin{lemma}\label{laplace codiff}
We have
\[
\dom\Delta = \left\{u\in\dom\eng:~\partial u \in \dom\partial^*\right\},
\]
and for $u\in\dom\Delta$,  $$\Delta u=-\partial^*\partial u .$$ 
\end{lemma}

\begin{proof}
Observe that $u\in \dom\mathcal{E}$ is in $\dom\Delta$ if and only if $\mathcal{E}(u,\phi) = -\int_X \phi \Delta u d\mu$ for all  $\phi\in W^{1,2}(X)$. Since
\[
\mathcal{E}(u,\phi)=\int_\mathcal{S} \partial u \partial \phi d\nu,
\]
the result follows.
\end{proof}

 The co-differential satisfies the following  inequality.

\begin{lemma}\label{poinc_div}
Let $\eta \in \mathrm{dom}\, \partial^*$ and  $x_0 \in X$.
Assume that $\int_X d(x_0,x) | \partial^* \eta | (x) d\mu (x)<\infty$ and  $d(x_0,x) | \eta (x) | \to 0$ as  $d(x_0,x) \to +\infty$. Then we have
\[
\int_\mathcal{S} | \eta | d\nu \le \int_X d(x_0,x) | \partial^* \eta| (x)  d\mu (x).
\]
\end{lemma}

\begin{proof}
We consider the sequence of cutoff functions $\{h_n\}_{n\ge 1}$ introduced in the proof of Lemma \ref{approxi L4}. Let now  $\eta \in \mathrm{dom}\, \partial^*$ and  $x_0 \in X$ satisfy the assumptions of the proposition. For $n \ge 1$, we define 
\begin{align*}
g_n (x)=h_n (x)\left(  \int_{\gamma(0,x)} \frac{\eta (y) }{|\eta (y) | +1/n}d\nu(y) -\int_{\gamma(0,x_0)} \frac{\eta (y) }{|\eta (y) | +1/n}d\nu(y) \right), \quad x \in \mathcal{S}
\end{align*}
which we extend by continuity to $X$. The extension is still denoted by $g_n$ and it is easy to see that it satisfies for every $x \in X$
\[
|g_n(x)| \le d(x_0,x),
\] 
because $|h_n(x)|\le 1$ and the function
\[
x\to \int_{\gamma(0,x)} \frac{\eta (y) }{|\eta (y) | +1/n}d\nu(y) -\int_{\gamma(0,x_0)} \frac{\eta (y) }{|\eta (y) | +1/n}d\nu(y)
\]
is 1-Lipschitz and vanishes at $x_0$.
Note also  that $g_n$ is compactly supported and therefore in $L^2(X,\mu)$. 

We have for $\nu$ a.e. $x\in X$
\[
\partial g_n (x)=\varepsilon_n(x)+\frac{ h_n(x)\eta (x) }{|\eta (x) | +1/n} .
\]
where
\[
\varepsilon_n(x)=\partial h_n (x) \left(\int_{\gamma(0,x)} \frac{\eta (y)}{|\eta (y) | +1/n}d\nu(y)-\int_{\gamma(0,x_0)} \frac{\eta (y) }{|\eta (y) | +1/n}d\nu(y)\right).
\]
Therefore
\begin{align*}
 \int_\mathcal{S}  \frac{ h_n(x)\eta (x)^2 }{|\eta (x) | +1/n} d\nu(x) &=\int_\mathcal{S}\partial g_n (x) \eta(x) d\nu(x) -\int_{\mathcal S}\varepsilon_n (x) \eta (x)  d\nu(x) \\
  &=\int_X g_n (x) \partial^* \eta(x) d\mu(x) -\int_{\mathcal S}\varepsilon_n (x) \eta (x)  d\nu(x) .
\end{align*}
This yields
\begin{align*}
\int_\mathcal{S}  \frac{ h_n(x)\eta (x)^2 }{|\eta (x) | +1/n} d\nu(x) & \le\int_X | g_n (x) \partial^* \eta(x) | d\mu(x)  +\int_{\mathcal S}|\varepsilon_n (x)| |\eta (x)|  d\nu(x) \\
 & \le \int_X d(x_0,x) |  \partial^* \eta(x) | d\mu(x) +\int_{\mathcal S} | \partial h_n (x)| d(x_0,x) |\eta (x) | d\nu(x),
\end{align*}
where we used in the last inequality that
\[
| \varepsilon_n (x) | \le | \partial h_n (x)| d(x_0,x). 
\]
%Recall that $h_n\to 1$ as $n\to \infty$. 
By Fatou's lemma, it therefore remains to prove that
\[
\lim_{n \to +\infty} \int_{\mathcal S} | \partial h_n (x)| d(x_0,x) |\eta (x) | d\nu(x)=0.
\]
This simply follows from the fact that
\begin{align*}
 \int_{\mathcal S} | \partial h_n (x)| d(x_0,x) |\eta (x) | d\nu(x) & \le \frac{1}{3^n}\int_{3^n\bar{W}_0\setminus 3^{n-1}\bar{W}_0 } d(x_0,x) |\eta (x) | d\nu(x) \\
 & \le \frac{\nu \left(3^n\bar{W}_0\setminus 3^{n-1}\bar{W}_0  \right)}{3^n} \sup_{ x \in 3^n\bar{W}_0\setminus 3^{n-1}\bar{W}_0} d(x_0,x) |\eta (x) | \\
 & \le C\sup_{ x \in 3^n\bar{W}_0\setminus 3^{n-1}\bar{W}_0} d(x_0,x) |\eta (x) |,
\end{align*}
where we used \eqref{eq:partial-hn} in the first inequality.
\end{proof}
From the previous lemmas, one can get the following $L^p$ bound for the gradient of the heat kernel which is interesting in itself.

\begin{lemma}\label{lemma int}
Let $p \ge 1$. There exists a constant $C>0$ such that for every $t >0$ and $x \in X$
\[
\int_\mathcal{S} |\partial_y p_t(x,y)|^p d\nu(y)\le  \frac{C}{t^{p-1/d_w}} .
\]
\end{lemma}

\begin{proof}
In view of Corollary \ref{gb heat 1}, it is enough to prove the result for $p=1$.  From Lemma \ref{laplace codiff} and Lemma \ref{poinc_div},
\[
\int_\mathcal{S} |\partial_y p_t(x,y)| d\nu(y) \le \int_X d(x,y) |\Delta_y p_t(x,y) | d\mu (y).
\]
We deduce then from Lemma \ref{sub_time} and \eqref{eq:boundpolyptbypct} that 
\begin{align*}
\int_\mathcal{S} |\partial_y p_t(x,y)| d\nu(y) 
& \le C t^{-1-\frac{d_{h}}{d_{w}}} \int_X d(x,y)  \exp\biggl(-c\Bigl(\frac{d(x,y)^{d_{w}}}{t}\Bigr)^{\frac{1}{d_{w}-1}}\biggr) d\mu(y) \\
& \le C t^{-2+\frac2{d_{w}}} \int_X   \exp\biggl(-c\Bigl(\frac{d(x,y)^{d_{w}}}{t}\Bigr)^{\frac{1}{d_{w}-1}}\biggr) d\mu(y).
\end{align*}
It is easy to see that
\[
\int_X \exp\biggl(-c\Bigl(\frac{d(x,y)^{d_{w}}}{t}\Bigr)^{\frac{1}{d_{w}-1}}\biggr) d\mu(y) \le C t^{\frac{d_h}{d_w}}
\]
and the results follows since $d_w=d_h+1$.
\end{proof}

We can now prove the continuity of the semigroup seen as an operator $ L^q(X,\mu) \to W^{1,p}(X)$.
\begin{theorem}\label{Lp bound gt}
Let $1 \le q \le p  \le +\infty$. There exists a constant $C>0$ such that for every $f \in L^q(X,\mu)$ and $t >0$,

\[
\|\partial P_t f \|_{L^p(\mathcal S,\nu)} \le \frac{C}{t^{\left( 1-\frac{1}{p}-\frac{1}{q} \right) \frac{1}{d_w}+\frac{1}{q}}} \|f\|_{L^q(X,\mu)}.
\]
In particular, $P_t :  L^p(X,\mu) \to W^{1,p}(X)$ is bounded for any $p \ge 1$ with 
\begin{equation}\label{eq:Lp-Wp}
\|  \partial P_t f \|_{L^p(\mathcal S,\nu)} \le \frac{C}{ t^{\alpha_p}} \|  f \|_{L^p(X,\mu)},
\end{equation}
where $\alpha_p=\left( 1-\frac{2}{p} \right) \frac{1}{d_w}+\frac{1}{p}$. 
\end{theorem}

\begin{proof}
Let $f \in L^1(X,\mu)$. 
For every adjacent vertices $u,v$ in some $V_n$ with $u \le v$, we have
\begin{align*}
P_tf(v)-P_tf (u)
%&=\int_X p_t(v,z) f(z) d\mu(z) -\int_X p_t(u,z) f(z) d\mu(z) \\
 & =\int_X (p_t(v,z)- p_t(u,z)) f(z) d\mu(z)
 \end{align*}
which yields
 \[
\int_{\mathbf{e}(u,v)} \partial P_tf \, d\nu=\int_X (p_t(v,z)- p_t(u,z)) f(z) d\mu(z).
 \]
 Since it is true for every $n$, $u,v \in V_n$, $u \le v$ and since
 \[
 |p_t(v,z)- p_t(u,z)| \le \int_{\mathbf{e}(u,v)} |\partial_y p_t (z,y) | d\nu (y),
 \]
 one deduces from Theorem 3.3 in \cite{BC} and Lemma \ref{lemma int} that 
\begin{equation}\label{eq:L1-W1}
\int_\mathcal{S} |\partial P_t f (x) |d\nu(x) \le\frac{C}{t^{d_h/d_w}}  \|  f \|_{L^1(X,\mu)}.
\end{equation}

On the other hand, for $f\in L^p(X,\mu)$ one has from \eqref{GB-semigroup}
\begin{align*}
 | \partial P_t f (x)|& \le    \frac{C}{t^{1/d_w}}  P_{ct}|f|(x) =  \frac{C}{t^{1/d_w}} \int_X p_{ct}(x,y) |f(y)| d\mu(y) \\
  & \le  \frac{C}{t^{1/d_w}} \left(\int_X p_{ct}(x,y)^{r} d\mu(y)\right)^{1/r} \left(\int_X |f(y)|^p d\mu(y) \right)^p\\
  &\le \frac{C}{t} \left(\int_X\exp\biggl(-c\Bigl(\frac{d(x,y)^{d_{w}}}{t}\Bigr)^{\frac{1}{d_{w}-1}}\biggr) d\mu(y) \right)^{1/r} \| f \|_{L^p(X,\mu)} \\
  & \le \frac{C}{t^{1-\frac{d_h}{rd_w}}} \| f \|_{L^p(X,\mu)}
\end{align*}
where $r$ is the conjugate exponent of $p$. We obtain therefore
\begin{equation}\label{eq:Lq-Winfty}
\|\partial P_t f \|_{L^\infty(\mathcal S,\nu)} \le  \frac{C}{t^{1-\left(1-\frac{1}{p} \right)\frac{d_h}{d_w}}} \| f \|_{L^p(X,\mu)}.    
\end{equation}
For  $p=\infty$, this bound also holds since 
\[
| \partial P_t f (x)| \le   \frac{C}{t^{1/d_w}}  P_{ct}|f|(x) \le \frac{C}{t^{1/d_w}}\|f\|_{L^{\infty}(X,\mu)}.
\]

In view of \eqref{eq:L1-W1} and  \eqref{eq:Lq-Winfty}  the conclusion follows then from the Riesz-Thorin interpolation theorem by noticing that $d_w-d_h=1$.
\end{proof}

\subsection{Sobolev  inequalities}

In this section we use the heat kernel gradient bounds from the previous sections to prove Sobolev inequalities.
We start with an interesting Poincar\'e type inequality for the heat kernel measures.

\begin{lemma}\label{lemma pseud point}
Let $p\ge 1$. There exist constants $c_1,c_2>0$ such that for every $f\in W^{1,p}(X)$, $t>0$ and $x \in X$
   \begin{align*}
\int_X  p_{t}(x,y) |f(x)-f(y) |^p d\mu(y) \le c_1  t^{ p \alpha_p} \int_{\mathcal{S}} p_{c_2 t}(x,z) | \partial f|^p (z) d\nu (z),
\end{align*}
where, as before, $\alpha_p=\left( 1-\frac{2}{p} \right) \frac{1}{d_w}+\frac{1}{p}.$
\end{lemma}

\begin{proof}
Let $f \in W^{1,p}(X)$. Similarly to \cite[Theorem 3.13]{BC}, we have for every $x,y \in X$
\[
| f(y)-f(x)| \le  d(x,y)^{1-\frac{1}{p}}  \left( \int_{\mathcal S \cap B(x,d(x,y)) } |\partial f|^p (z) d\nu(z)\right)^{1/p}.
\]
Therefore, in view of \eqref{eq:boundpolyptbypct} we obtain
\begin{align*}
\int_X  p_{t}(x,y) |f(x)-f(y) |^p d\mu(y)  & \le \int_X  \int_{\mathcal S \cap B(x,d(x,y)) } p_{t}(x,y) d(x,y)^{p-1}  |\partial f|^p (z) d\nu(z) d\mu(y)  \\
 & \le  c_3 t^{  \frac{p-1}{d_w}}  \int_X  \int_{\mathcal S \cap B(x,d(x,y)) } p_{c_4t}(x,y)  |\partial f|^p (z) d\nu(z) d\mu(y) \\
 & \le c_3 t^{  \frac{p-1}{d_w}} \int_{\mathcal S  }\left( \int_{X \setminus B(x,d(x,z))}  p_{c_4t}(x,y) d\mu(y) \right) |\partial f|^p (z) d\nu(z) .
\end{align*}

It remains to bound $\int_{X \setminus B(x,d(x,z))}  p_{c_4t}(x,y) d\mu(y)$.  First observe that for any $r\ge 0$ and $c>0$, there exists $C>0$ such that 
\[
\int_{X\setminus B(x,r)}\exp\biggl(-c\Bigl(\frac{d(x,y)^{d_w}}{t}\Bigr)^{\frac{1}{d_w-1}}\biggr) \,d\mu(y) 
\le C t^{\frac{d_h}{d_w}} \exp\biggl(-\frac{c}2\Bigl(\frac{r^{d_w}}{t}\Bigr)^{\frac1{d_w-1}}\biggr).
\]
See for instance Step 1 for the proof of \cite[Lemma 3.2]{Chen}. 
Denoting now $r=d(x,z)$, we conclude the estimate below
\begin{align*}
    \int_{X \setminus B(x,r)}  p_{c_4t}(x,y)  d\mu(y) &\leq\frac{c_{5}}{t^{d_h/d_w}} \int_{X\setminus B(x,r)}\exp\biggl(-c_{6}\Bigl(\frac{d(x,y)^{d_w}}{t}\Bigr)^{\frac{1}{d_w-1}}\biggr) \,d\mu(y)
    \\ &\le 
    c_7 \exp\biggl(-\frac{c_{6}}{2}\Bigl(\frac{r^{d_w}}{t}\Bigr)^{\frac{1}{d_w-1}}\biggr)
     \\ &\le 
     c_8  t^{\frac{d_h}{d_w}} p_{c_9t}(x,z).
\end{align*}
The proof is completed by recalling that $d_w=d_h+1$.
\end{proof}

The following corollary is an $L^p$ pseudo-Poincar\'e inequality for the heat semigroup.
\begin{lemma}\label{pseudo poincare 2}
Let $p\ge 1$. There exists a constant $C>0$ such that for every $f\in W^{1,p}(X)$ and $t>0$
\[
\| P_tf -f \|_{L^p(X,\mu)}\le C t^{\alpha_p} \| \partial f \|_{L^p(\mathcal S,\nu)}.
\]
\end{lemma}

\begin{proof}
From Lemma \ref{lemma pseud point} and H\"older's inequality, we have
\begin{align*}
\| P_t f -f \|^p_{L^p(X,\mu)} 
%& =\int_X |P_tf(x)-f(x)|^p d\mu(x) \\
% &=\int_X \left|\int_X p_t(x,y)f(y) d\mu(y) -f(x)\right|^p d\mu(x) \\
 &=\int_X \left|\int_X p_t(x,y)(f(y)-f(x)) d\mu(y) \right|^p d\mu(x) \\
 &\le \int_X \int_X p_t(x,y) |f(y)-f(x)|^p d\mu(y) d\mu(x) \\
 &\le C t^{p\alpha_p} \int_X \int_{\mathcal{S}} p_{c t}(x,z) | \partial f(z)|^p  d\nu (z) d\mu(x) \\
 &\le C t^{p\alpha_p} \int_{\mathcal{S}} | \partial f(z)|^p  d\nu (z),
\end{align*}
where we used Fubini's theorem in the last step.
\end{proof}

We are now ready for applications to Sobolev inequalities.

\begin{theorem}\label{Nash Vicsek} 
For $p>1$, the following Nash inequality holds for every $f \in W^{1,p}(X)$,
\[
\| f \|_{L^p(X,\mu)} \le C \| \partial f \|_{L^p(\mathcal S,\nu)}^{\theta} \| f \|^{1-\theta}_{L^1(X,\mu)}
\]
where $\theta=\frac{(p-1)d_h}{p-1+pd_h}$. 
\end{theorem}

\begin{proof}
We first note the following two easily proved properties which indicate that the Sobolev norm behaves nicely under cutoffs. In what follows we denote $a\wedge b=\min (a,b)$ and $a_+=\max (a,0)$.
\begin{itemize}
\item For every $s,t \ge 0$, and $f\in W^{1,p}(X)$,  
\[
\int_\mathcal{S} |\partial((f-t)_+ \wedge s )|^p d\nu \le \int_\mathcal{S} | \partial f|^p d\nu.
\]
\item  For any non-negative $f \in W^{1,p}(X)$  and any $\rho>1$,
\[
 \sum_{k \in \mathbb{Z}} \int_\mathcal{S} |\partial f_{\rho,k}|^p d\nu \le  \int_\mathcal{S} |\partial f|^p d\nu,
\]
where $f_{\rho,k}:=(f-\rho^k)_+ \wedge \rho^k(\rho-1)$, $k \in \mathbb{Z}$. 
\end{itemize}
On the other hand, it immediately follows from the heat kernel upper bound that for any $f \in L^1(X,\mu)$ one has
\[
\| P_t f \|_{L^\infty(X,\mu)} \le \frac{C}{t^{d_h/d_w}} \| f \|_{L^1(X,\mu)}
\]
Remarkably, together with Lemma \ref{pseudo poincare 2}, the previous observations are enough to obtain the full scale of Gagliardo-Nirenberg inequalities and in particular the stated Nash inequalities. The results follow from applying the results of \cite[Theorem 9.1]{MR1386760} (see also \cite{MR4196573}).

\end{proof}

\begin{remark}
For $p=2$, Nash inequalities  have been studied in connection with heat kernel estimates in the more general context of p.c.f. fractals. In particular, in the case $p=2$, Theorem \ref{Nash Vicsek}  recovers \cite[Theorem 8.3]{Barlow} on the Vicsek set.
\end{remark}

One can obtain further Sobolev type inequalities from the previous results.

\begin{theorem}\label{frac Riesz}
For $p \ge 1$, $0<s<\alpha_p$, there exists a constant $C>0$ such that for every $f\in W^{1,p}(X)$
\begin{align}\label{MI1}
\|(-\Delta)^s f\|_{L^p(X,\mu)}\le C \| f\|^{1-\frac{s}{\alpha_p}}_{L^p(X,\mu)} \| \partial f\|_{L^p(\mathcal S,\nu)}^{\frac{s}{\alpha_p}}.
\end{align}
For $p \ge 1$, $s>\alpha_p$, let $f\in L^p(X, \mu)$ be such that $(-\Delta)^s f\in L^p(X, \mu)$. Then, $f \in W^{1,p}(X)$ and there exists a constant $C>0$ such that 
\begin{align}\label{MI2}
\| \partial f \|_{L^p(\mathcal S,\nu)} \le C \| f \|_{L^p(X,\mu)}^{1-\frac{\alpha_p}{s}} \| (-\Delta)^sf \|_{L^p(X,\mu)}^{\frac{\alpha_p}{s}}.
\end{align}  
\end{theorem}
\begin{proof}
  The inequality \eqref{MI1} follows from  Lemma \ref{pseudo poincare 2} and the proof of \cite[Proposition 4.23]{MR4075578}.
%  , that is, 
 %\[
%  \|(-\Delta)^s f\|_{L^p(X,\mu)}\le C \| f\|^{1-\frac{s}{\alpha_p}}_{L^p(X,\mu)} \left(\sup_{t>0}t^{-\alpha_p} \left(\int_X \int_X p_t(x,y) |f(y)-f(x)|^p d\mu(y) d\mu(x)\right)^{1/p}\right)^{\frac{s}{\alpha_p}}.
%\]
For the inequality \eqref{MI2}, we adapt the approach of the proof of \cite[Lemma 4.5]{CoulhonSikora}.  For any $s>\alpha_p$, Theorem \ref{Lp bound gt} yields that for every $f\in L^p(X, \mu)$
\[
\|\partial(I-t\Delta)^{-s} f\|_{L^p(\mathcal S,\nu)} \le C t^{-\alpha_p} \| f\|_{L^p(X,\mu)}.
\]
Indeed, this can be observed by writing
\[
\partial (I-t\Delta)^{-s}f =\frac1{\Gamma(s)} \int_0^{\infty} r^{s-1}e^{-r}\partial P_{rt}f dr.
\]
Hence we have
\[
\|\partial f\|_{L^p(\mathcal S,\nu)} 
\le C t^{-\alpha_p}\|(I-t\Delta)^{s}f\|_{L^p(X,\mu)}
\le C t^{-\alpha_p}\left(\|f\|_{L^p(X,\mu)}+t^s\|(-\Delta)^s f\|_{L^p(X,\mu)}\right),
\]
where the second inequality follows from the fact that the operator $(I-t\Delta)^s(I+(-t\Delta)^s)^{-1}$ is bounded on $L^p(X,\mu)$ by analyticity.
The proof is concluded  by choosing
\[
t=\|f\|_{L^p(X,\mu)}^{\frac1s} \|(-\Delta)^s f\|_{L^p(X,\mu)}^{-\frac1s}.
\]
\end{proof}

The case $s=\alpha_p$ in \eqref{MI1} and \eqref{MI2} is left open, and  the following open question is therefore natural:

\begin{conjecture}[Boundedness of the Riesz transform]
Does the following hold: The operator $\partial (-\Delta)^{-\alpha_1}$ is of weak type $(1,1)$ and for $p >1$ there exist constants $c,C>0$ such that for every $f \in W^{1,p}(X)$,
\[
c \| \partial f \|_{L^p(\mathcal S,\nu)} \le \| (-\Delta)^{\alpha_p} f \|_{L^p(X,\mu)} \le C  \| \partial f \|_{L^p(\mathcal S,\nu)} \, ?
\]
\end{conjecture}

\section{Hodge Laplacian and associated heat kernel}

The goal of the section is to introduce a Hodge Laplacian and a Hodge semigroup and to prove that the associated heat kernel admits a sub-Gaussian upper bound.  

\subsection{Hodge Laplacian}

In view of the  map $\partial: W^{1,p}(X,\mu) \to L^p(\mathcal S,\nu)$ where $\partial$ is thought of as a differential, it is natural to think of elements of $L^p(\mathcal S,\nu)$ as $p$ integrable one-forms, see \cite{MR3896108,MR1986156,MR2964679} for related discussions in the general setting of Dirichlet spaces. With this in mind, we define the Hodge Laplacian on $X$ by $\vec\Delta =- \partial\partial^*$ with the domain
\[
\dom\vec\Delta = \left\{\eta \in L^2(\mathcal S,\nu)~|~\partial^* \eta \in \dom\partial=W^{1,2}(X)\right\}.
\]
As already pointed out in Section \ref{sec:3.3}, $\partial^*$ is a densely defined and closed operator. Therefore, from a Von Neumann's theorem (confer \cite[Theorem 8.4]{Tay} or the proof of Theorem VIII.32 in \cite{RS72}),  the operator $\vec\Delta=- (\partial^*)^*\partial^*$ is self-adjoint in $L^2(\mathcal S,\nu)$. Alternatively, since $\partial^*$ is closed, we may also see $\vec\Delta$  as the self-adjoint generator of the closed symmetric bilinear form on $L^2(\mathcal S,\nu)$:
\[
\vec{\eng} (\omega, \eta)=\int_X \partial^* \omega \partial^* \eta d\mu.
\]

 It is worth noting that $\vec{\eng}$ is not a Dirichlet form (it is not Markovian). We shall denote the Hodge semigroup  $e^{t\vec \Delta}$ in $L^2(\mathcal S,\nu)$ by $\vec{P}_t$. It is classically  defined via the spectral theorem by using functional calculus (see for instance Theorem 4.15 in  \cite{Bau14}). The fundamental property of the Hodge semigroup is the following intertwining formula.

\begin{theorem}\label{intertwining1}
For $f \in  W^{1,2}(X)$, we have
\[
  \partial P_t f=\vec{P}_t \partial f, \quad t \ge 0.
\]
\end{theorem}

\begin{proof}
Since we have $\partial \Delta =\vec{\Delta} \partial $ on the space $\cap_{k \ge 1} \mathrm{dom}(\Delta^k)$, this type of commutation  is standard, see Theorem 3.1 in  Shigekawa \cite{Shi00} for statement in a more general abstract setting. Indeed, for $f \in L^2(X,\mu)$ and $t >0$, one has from spectral theory that $\Delta P_t f \in W^{1,2}(X)$. We have then for $t>0$
\[
\frac{d }{d t} \partial P_t f =\partial  \Delta P_t f =\vec{\Delta} \partial P_t f.
\]
Here, we see $t \to \partial P_t f$ as a curve $\mathbb{R}_{\ge 0} \to L^2(\mathcal S,\nu)$. Denote now $\Psi (t)=\partial P_t f  -\vec{P}_t \partial f \in \mathrm{dom} (\vec \Delta)$. We note that
\[
\frac{d }{d t}\Psi(t)=\vec \Delta \Psi (t), \quad \Psi (0)=0.
\]
 Let $\alpha(t)= \int_\mathcal{S} \Psi (t)^2 d\nu$. We have  $$\alpha'(t)= 2\int_\mathcal{S} \Psi (t) \frac{d }{d t} \Psi(t)d\nu =2\int_\mathcal{S} \Psi(t) \vec \Delta \Psi (t) d\nu=-2 \int_\mathcal{S} \partial^* \Psi (t)^2 d\nu \le 0.$$ Therefore $\Psi$ is non-increasing and non-negative. Since $\Psi (0)=0$, we get that $\Psi (t)=0$ for every $t \ge 0$.
\end{proof}

\subsection{Hodge heat kernel upper  bound}

Our goal in this section is to prove the following result.

\begin{theorem}\label{bound hodge2}
The Hodge semigroup $\vec{P}_t$ admits a heat kernel, that is, for every $t>0$ there exists a measurable function $\vec{p}_t (x,y):  \mathcal S \times \mathcal S \to \mathbb R$ such that for every $\eta \in L^2 (\mathcal S,\nu)$, and $\nu$ a.e. $x \in \mathcal {S}$,
\[
\vec{P}_t \eta (x)=\int_\mathcal S \vec{p}_t (x,y) \eta(y) d\nu (y).
\]
Moreover, there exist $c_1,c_2>0$ such that for every $t >0$ and $\nu \otimes \nu$-a.e. $x,y \in \mathcal S$
    \[
 | \vec{p}_t (x,y)| \le  \frac{c_1}{t^{1/d_w}}  \exp\biggl(-c_2\Bigl(\frac{d(x,y)^{d_{w}}}{t}\Bigr)^{\frac{1}{d_{w}-1}}\biggr)
 \]
 and there exists a constant $c_3>0$ such that for every $t >0$ and $\nu$-a.e $x \in \mathcal{S}$
 \[
  \int_{\mathcal S} | \vec{p}_t (x,y)| d\nu(y) \le c_3.
 \]
\end{theorem}

We split the proof into several lemmas.

\begin{lemma}\label{bound funda fg}
 Let $p \ge 1$. There exist constants $c,C>0$ such that for every $t >0$, $f \in W^{1,p}(X)$, and $\nu$ a.e. $x \in \mathcal S$,
\[
|\partial P_t f(x) |^p\le C t^{1-\frac{2}{d_w}} \int_\mathcal{S} p_{ct}(x,y) |\partial f (y)|^p  d\nu(y).
\]
\end{lemma}

\begin{proof}
Let $f \in W^{1,p}(X)$. From Remark \ref{remark median} we  get that for $\nu$ a.e. $x\in \mathcal S$
\begin{align*}
|\partial P_t f(x)|&   \le  \frac{C}{t^{1/d_w}}   \int_X p_{ct}(x,y) |f(x)-f(y) | d\mu(y) \\
 &\le  \frac{C}{t^{1/d_w}}  \left(  \int_X p_{ct}(x,y) |f(x)-f(y) |^p d\mu(y)  \right)^{1/p}
\end{align*}
and we conclude the proof by applying Lemma \ref{lemma pseud point}.
\end{proof}

\begin{lemma}\label{cont r3}
Let $p \ge 1$. There exist constants $c,C>0$ such that for every $t >0$, $\eta \in L^1(\mathcal S,\nu) \cap L^\infty(\mathcal S,\nu)$, and $\nu$ a.e. $x \in \mathcal S$,
\[
|\vec P_t \eta (x) |^p\le C t^{1-\frac{2}{d_w}} \int_\mathcal{S} p_{ct}(x,y) |\eta (y)|^p d\nu(y).
\]
\end{lemma}

\begin{proof}
Let $p \ge 1$. In view of Theorem \ref{intertwining1} and Lemma \ref{bound funda fg},  we  obtain that for every $f \in W^{1,p}(X)\cap W^{1,2}(X)$ there holds for $\nu$ a.e. $x \in \mathcal S$,
\begin{align}\label{eq: ref con1}
| \vec{P}_t \partial f (x)|^p\le c_1 t^{1-\frac{2}{d_w}} \int_\mathcal{S} p_{c_2t}(x,y) |\partial f(y)| ^p d\nu(y) .
\end{align}    
In particular, from the heat kernel upper bound in \eqref{eq:sub-Gaussian1} one has for every $f \in W^{1,p}(X)\cap W^{1,2}(X)$
\begin{align}\label{eq: ref con2}
\| \vec{P}_t \partial f \|_{L^\infty(\mathcal S,\nu)}^p\le \frac{C}{t^{1/d_w}} \| \partial f \|^p_{L^p(\mathcal S, \nu)} .
\end{align}  
With $p=2$, using the continuity of $\vec P_t :L^2(\mathcal S, \nu)\to L^2(\mathcal S, \nu)$ and the fact that $\partial: W^{1,2}(X) \to   L^2(\mathcal S,\nu) $ has a dense range (see Proposition \ref{dense range}), we deduce first that  $\vec{P}_t$ is a bounded operator $L^2(\mathcal S, \nu) \to L^\infty(\mathcal S,\nu)$ with
\[
\| \vec{P}_t \|^2_{L^2 \to L^\infty}\le \frac{C}{t^{1/d_w}}.
\]

Let now $\eta \in L^1(\mathcal S,\nu) \cap L^\infty(\mathcal S,\nu)$ and consider the approximating sequence $\phi_n$ defined in \ Lemma \ref{approxi L4}. We have for every $n$ and $\nu$ a.e. $x \in \mathcal S$
\begin{align}\label{pol th}
| \vec{P}_t \partial \phi_n (x)|^p\le c_1 t^{1-\frac{2}{d_w}} \int_\mathcal{S} p_{c_2t}(x,y) |\partial \phi_n(y)| ^p d\nu(y). 
\end{align}
Since $\partial \phi_n$ converges to $\eta$ in $L^2(\mathcal S,\nu)$, from the continuity of $\vec{P}_t :L^2(\mathcal S, \nu) \to L^\infty(\mathcal S,\nu)$, we get that the left hand side in \eqref{pol th} $\nu$ a.e. converges to $| \vec{P}_t \eta (x)|^p$. On the other hand, the right hand side of \eqref{pol th} $\nu$ a.e. converges to $c_1 t^{1-\frac{2}{d_w}} \int_\mathcal{S} p_{c_2t}(x,y) |\eta(y)| ^p d\nu(y)$. For $p>1$, this simply follows from the fact that $\partial \phi_n$ converges to $\eta$ in $L^p(\mathcal S,\nu)$ and for $p=1$ this can be checked using \eqref{eq:partial phi} and the estimate
\begin{align*}
& \int_\mathcal{S} p_{c_2t}(x,y) \left|\partial h_n (y)\int_{\gamma (0,y)} \eta d\nu \right| d\nu(y)  \\
\le & \int_\mathcal{S} 3^{-n} 1_{3^n\bar{W}_0\setminus 3^{n-1}\bar{W}_0 }(y) p_{c_2t}(x,y) d\nu(y)  \|\eta \|_{L^1(\mathcal S,\nu)}
\\ 
\le & C t^{-\frac{d_h}{d_w}} 3^{-n} \|\eta \|_{L^1(\mathcal S,\nu)} \int_{3^n\bar{W}_0\setminus 3^{n-1}\bar{W}_0} \exp\biggl(-C\Bigl(\frac{d(x,y)^{d_{w}}}{t}\Bigr)^{\frac{1}{d_{w}-1}}\biggr)d\nu(y) \\
& \longrightarrow_{n \to +\infty}  0.
\end{align*}
One concludes that for $\nu$ a.e. $x \in \mathcal S$,
\[
|\vec P_t \eta (x) |^p\le C t^{1-\frac{2}{d_w}} \int_\mathcal{S} p_{ct}(x,y) |\eta (y)|^p d\nu(y).
\]

\end{proof}

\begin{lemma}
 For every $t>0$ there exists a measurable function $\vec{p}_t(x,y)$ on $\mathcal{S}\times \mathcal{S}$ with
\[
|\vec{p}_t(x,y)| \le \frac{c_1}{t^{1/d_w}} \exp\biggl(-c_2\Bigl(\frac{d(x,y)^{d_{w}}}{t}\Bigr)^{\frac{1}{d_{w}-1}}\biggr)
\]
for $\nu \otimes \nu$ a.e. $(x,y) \in \mathcal{S} \times \mathcal{S}$,  such that for every $\eta\in L^1(\mathcal S,\nu)\cap L^2(\mathcal S,\nu)$ and $\nu$ a.e. $x \in \mathcal S$
\[
\vec{P}_t \eta (x)=\int_\mathcal S \vec{p}_t (x,y) \eta(y) d\nu (y).
\]
\end{lemma}

\begin{proof}
From Lemma \ref{cont r3}, $\vec{P}_t$ extends from $L^1(\mathcal{S},\nu) \cap L^\infty(\mathcal{S},\nu) $ to $L^1(\mathcal{S},\nu)$ as a bounded operator $L^1(\mathcal{S},\nu) \to L^\infty(\mathcal{S},\nu)$, still denoted by $\vec{P}_t$, that satisfies
\[
\| \vec{P}_t \|_{L^1 \to L^\infty}\le \frac{C}{t^{1/d_w}}.
\]
From Proposition 1.2.5 in \cite{BGL}, we deduce that for every $t>0$ there exists a measurable function $\vec{p}_t(x,y)$ on $\mathcal{S}\times \mathcal{S}$ with
\[
|\vec{p}_t(x,y)| \le \frac{C}{t^{1/d_w}}
\]
for $\nu \otimes \nu$ a.e. $(x,y) \in \mathcal{S} \times \mathcal{S}$ such that for every $\eta\in L^1(\mathcal S,\nu)$ and $\nu$ a.e. $x \in \mathcal S$
\[
\vec{P}_t \eta (x)=\int_\mathcal{S} \vec{p}_t(x,y) \eta(y) d\nu(y).
\]
From Lemma \ref{cont r3}, we infer  that for every $\eta \in  L^1(\mathcal S,\nu) \cap  L^\infty(\mathcal S,\nu) $
\[
\left |\int_\mathcal{S} \vec{p}_t (x,y)  \eta  (y) d\nu(y) \right| \le  C t^{1-\frac{2}{d_w}} \int_\mathcal{S} p_{ct}(x,y) | \eta | (y) d\nu(y)
\]
and the conclusion follows by testing this inequality for all $\eta=1_{A}$, where $A$ is any finite length subset of $\mathcal S$.

\end{proof}

\begin{lemma}\label{lemma grad pt}
There exists a constant $C>0$ such that for every $t \ge 0$,  and $f \in W^{1,\infty} (X)$
\[
\| \partial P_t f \|_{L^\infty (\mathcal S,\nu)} \le C  \| \partial  f \|_{L^\infty (\mathcal S,\nu)}.
\]
\end{lemma}

\begin{proof}
Let  $f\in  W^{1,\infty}(X)$. From Remark \ref{remark median} we have for any $\nu$ a.e. $x\in \mathcal S$, 
 \[
 | \partial P_t f (x)| \le    \frac{C}{t^{1/d_w}}  P_{ct}(|f-f(x)|)(x).
 \]
 We now  have
 \[
  P_{ct}(|f-f(x)|)(x)=\int_X p_{ct} (x,y) |f(y)-f(x) | d\mu(y).
 \]
 However, since $f$ is Lipschitz  we have $|f(y)-f(x) | \le d(x,y)\| \partial  f \|_{L^\infty (\mathcal S,\nu)} $, therefore we obtain
 \[
 | \partial P_t f (x)| \le    \frac{C}{t^{1/d_w}}  \int_X p_{ct} (x,y) d(x,y) d\mu(y)  \, \| \partial  f \|_{L^\infty (\mathcal S,\nu)}.
 \]
From \eqref{eq:boundpolyptbypct} one has $\int_X p_{ct} (x,y) d(x,y) d\mu(y)  \le C t^{1/d_w}$ and the conclusion follows.  
\end{proof}

\begin{lemma}
    There exists a constant $C>0$ such that for every $t >0$ and $\nu$-a.e $x \in \mathcal{S}$
 \[
  \int_{\mathcal S} | \vec{p}_t (x,y)| d\nu(y) \le C.
 \]
\end{lemma}

\begin{proof}
From Theorem \ref{intertwining1} and Lemma \ref{lemma grad pt} one has for every $f \in W^{1,2}(X) \cap W^{1,\infty}(X)  $
\[
\| \vec {P}_t \partial f \|_{L^\infty(\mathcal S,\nu)} \le C  \| \partial f \|_{L^\infty(\mathcal S,\nu)}.
\]
Let now $\eta \in L^2(\mathcal S,\nu) \cap L^\infty(\mathcal S,\nu)$ and consider the approximating sequence $\phi_n$ of Lemma \ref{approxi L4}. We have  $\vec {P}_t \partial \phi_n$ converges to $\vec {P}_t \eta$ in $L^\infty (\mathcal S,\nu)$ when $n \to +\infty$. This yields
 \[
 \| \vec P_t \eta \|_{L^\infty(\mathcal S,\nu)} \le C \|  \eta \|_{L^\infty(\mathcal S,\nu)}.
 \]
 Therefore for any set $A \subset \mathcal S$ of finite length one has for $\nu$ a.e. $x\in \mathcal S$
 \[
\left| \int_A \vec{p}_t (x,y)  d\nu(y) \right| \le C
 \]
and the result follows.
\end{proof}

\section{Heat kernel gradient  bounds in the  compact Vicsek set}

In this section we discuss gradient bounds for the heat kernel on the compact Vicsek set $K$. Most of the proofs are simpler and simple modifications of  arguments of the previous sections. We denote by $\mathcal S_K=K \cap \mathcal S$ the skeleton of $K$ and otherwise use the notations already introduced in Section \ref{notations Vicsek}. Let $1 \le p\le \infty$. For $f\in C(K)$, we say that $f\in W^{1,p}(K)$ if $f$ can be extended to a function $\tilde f \in W^{1,p}(X)$ on $X$. For $\nu$ a.e. $x \in \mathcal S_K$, we then set $\partial f (x) =\partial \tilde f (x)$. For $f,g \in W^{1,2}(K)$, we define
\[
\mathcal{E}_K(f,g)=\int_{\mathcal S_K} \partial f \,  \partial g \, d\nu.
\]
The following result follows from \cite[Theorem 7.14]{Barlow}.
\begin{theorem}
The quadratic form $\mathcal{E}_K$ with domain $W^{1,2}(K)$ is a strongly local and regular Dirichlet form in $L^2(K,\mu)$. A core for $\mathcal{E}_K$ is the set of piecewise affine functions defined on $K$. Moreover for every $f \in W^{1,2}(K)$
\[
\mathcal{E}_K(f,f)=\lim_{m \to +\infty} \frac{1}{2} 3^{m}  \sum_{x,y \in W_m, x \sim y} |f(x)-f(y)|^2.
\]
\end{theorem}

 Let $\Delta_K $ be the generator of $(\mathcal E_K,W^{1,2}(K))$  on $L^2(K, \mu)$. The associated heat semigroup $(P^K_t)_{t\ge 0}$ admits a heat kernel  that we denote by $p^K_t(x,y)$. 
A difference with respect to the unbounded case is that from spectral theory, the  heat kernel $p_t^K(x,y)$ admits a convergent spectral expansion:
\begin{equation}\label{eq:HKexpansion}
p_t^K(x,y)=1+\sum_{j=1}^{+\infty} e^{-\lambda_j t} \Phi_j(x) \Phi_j(y), \quad t>0, x,y \in K,
\end{equation}
where the $\lambda_j$'s are the eigenvalues of $-\Delta_K$ and the $\Phi_j$'s the corresponding eigenfunctions.

From  \cite[Theorem 8.18]{Barlow} the  heat kernel $p^K_t(x,y)$   satisfies sub-Gaussian estimates in small times.  More precisely, for every $(x,y)\in K\times K$ and $t\in (0,1)$
\begin{equation}\label{eq:sub-Gaussian12}
c_1t^{-\frac{d_{h}}{d_{w}}}\exp\biggl(-c_2\Bigl(\frac{d(x,y)^{d_{w}}}{t}\Bigr)^{\frac{1}{d_{w}-1}}\biggr) \le p^K_t(x,y) 
\le  c_3 t^{-\frac{d_{h}}{d_{w}}}\exp\biggl(-c_4\Bigl(\frac{d(x,y)^{d_{w}}}{t}\Bigr)^{\frac{1}{d_{w}-1}}\biggr).
\end{equation}

\begin{lemma}\label{sub_time2_K}
 There exist constants $c,C>0$ such that for every $t >0$,
\begin{equation}\label{eq:sub-Gaussian2_K}
|\partial_t p^K_t(x,y) |
\le    \frac{C e^{-\lambda_1  t}}{ (1 \wedge t)^{1+d_{h}/d_{w}}}\exp\biggl(-c\Bigl(\frac{d(x,y)^{d_{w}}}{t}\Bigr)^{\frac{1}{d_{w}-1}}\biggr).
\end{equation}
\end{lemma}

\begin{proof}
For small time $t\in (0,1)$, the bound
\[
|\partial_t p^K_t(x,y) |
\le  \frac{C}{ t^{1+d_{h}/d_{w}}}\exp\biggl(-c\Bigl(\frac{d(x,y)^{d_{w}}}{t}\Bigr)^{\frac{1}{d_{w}-1}}\biggr)
\]
follows from the heat kernel estimates \eqref{eq:sub-Gaussian12}, while for $t \ge 1$ the estimate
\[
|\partial_t p^K_t(x,y) |\le Ce^{-\lambda_1 t}
\]
follows from the spectral expansion \eqref{eq:HKexpansion}. Indeed, one gets from \eqref{eq:HKexpansion}
\[
|\partial_t p^K_t(x,y) |\le \sum_{j=1}^{+\infty}  \lambda_j e^{-\lambda_j t} \| \Phi_j \|^2_{L^\infty (K,\mu)}.
\]
Now $\| \Phi_j \|_{L^\infty (K,\mu)}$ can be bounded by noticing that for any $0<s<1$ 
\[
|\Phi_j(x)|=e^{\lambda_j s} | P_s \Phi_j (x)|=e^{\lambda_j s} \left|\int_K p_s^K (x,y) \Phi_j (y)d\mu (y)\right| \le e^{\lambda_j s}  \sqrt{p_{2s}^K (x,x)}\le C\frac{e^{\lambda_j s}}{s^{d_h/2d_w}},
\]
where we used the sub-Gaussian upper bound \eqref{eq:sub-Gaussian12} in the last inequality.
Choosing $s=\frac{\lambda_1}{\lambda_j}$ yields
\begin{align}\label{boundEigenfunctio}
\| \Phi_j \|_{L^\infty (K,\mu)} \le C \lambda_j^{d_h/2d_w}
\end{align}
and therefore for $t \ge 1$
\[
|\partial_t p^K_t(x,y) |\le C  \sum_{j=1}^{+\infty}  \lambda_j^{1+d_h/d_w} e^{-\lambda_j t} .
\]
The conclusion follows then easily since
\[
\sum_{j=1}^{+\infty}  \lambda_j^{1+d_h/d_w} e^{-\lambda_j t} \le C e^{-\lambda_1 t}.
\]
\end{proof}

The following gradient bound can be proved as Proposition \ref{gb heat 1}.
 
\begin{theorem}
There exist $c,C>0$ such that for every $t >0$, $y \in K$, and $\nu$ a.e $x \in \mathcal S_K$
    \[
 | \partial_x p^K_t(x,y)| \le C \frac{ e^{-\lambda_1 t} }{ (1 \wedge t)^{1/d_w}}  p^K_{ct}(x,y).
 \]
\end{theorem}

\begin{proof}
The proof of the estimate
\begin{align}\label{startestimate}
| \partial_x p^K_t(x,y)| \le  \frac{C}{ t^{1/d_w}}  p^K_{ct}(x,y), \quad t >0,
\end{align}
is an appropriate modification of that of Proposition \ref{gb heat 1}, so we omit the details and focus on the exponential decay when $t \to +\infty$. Since for $t \ge 1$  we uniformly have from the spectral expansion \eqref{eq:HKexpansion} that 
 \[
 p^K_{ct}(x,y) \ge c_1
 \]
for some constant $c_1>0$, it is enough to prove that for $t\ge 1$
 \[
|\partial_x p^K_t(x,y) |\le Ce^{-\lambda_1 t}.
\]
To prove the latter we observe that for $x,y,z \in K$
\begin{align*}
 |p^K_t(x,z)-p^K_t(y,z) | & \le \sum_{j=1}^{+\infty} e^{-\lambda_j t} | \Phi_j(x)-\Phi_j(y)| |\Phi_j (z)| \\
 & \le \sum_{j=1}^{+\infty} e^{-\lambda_j t} \| \Phi_j \|_{L^\infty(K,\mu)} |\Phi_j(x)-\Phi_j(y)| \\
 & \le C\sum_{j=1}^{+\infty} \lambda_j^{d_h/2d_w} e^{-\lambda_j t}  |\Phi_j(x)-\Phi_j(y)|
\end{align*}
where we used \eqref{boundEigenfunctio} in the last inequality. On the other hand, we have for any $s>0$
\begin{align*}
|\Phi_j(x)-\Phi_j(y)|&\le e^{\lambda_j s}\int_K |p_s^K (z,x)-p_s^K (z,y)| | \Phi_j (z)| d\mu(z) \\
 &\le C \lambda_j^{d_h/2d_w} e^{\lambda_j s} \int_K |p_s^K (z,x)-p_s^K (z,y)|  d\mu(z) \\
 &\le C \lambda_j^{d_h/2d_w} e^{\lambda_j s} \frac{d(x,y)}{s^{1/d_w}},
\end{align*}
where we used \eqref{startestimate} in the last inequality. With $s=1/\lambda_j$ this gives
\[
|\Phi_j(x)-\Phi_j(y)|\le C \lambda_j^{(d_h+2)/2d_w} d(x,y)
\]
so that
\[
|p^K_t(x,z)-p^K_t(y,z) |\le Cd(x,y) \sum_{j=1}^{+\infty} \lambda_j^{(d_h+1)/d_w}e^{-\lambda_j t}
\]
and the proof is now complete.
\end{proof}

The following result is the analogue of Theorem \ref{eq:Lp-Wp} and can be proved in an identical way.

\begin{theorem}\label{Lp bound gt compact}
Let $1 \le q \le p  \le +\infty$. There exists a constant $C>0$ such that for every $f \in L^q(K,\mu)$ and $t >0$,

\[
\|\partial P^K_t f \|_{L^p(\mathcal S_K,\nu)} \le \frac{Ce^{-\lambda_1 t}}{(1 \wedge t)^{\left( 1-\frac{1}{p}-\frac{1}{q} \right) \frac{1}{d_w}+\frac{1}{q}}} \|f\|_{L^q(K,\mu)}.
\] 
\end{theorem}

As before, the co-differential is defined as the adjoint of the weak gradient $\partial$. More precisely, the operator $\partial^*$ is the densely defined operator from $L^2(\mathcal{S}_K,\nu) \to L^2(K,\mu)$ with domain
\[
\mathrm{dom} \, \partial^*:= \left\{\eta \in L^2(\mathcal{S}_K,\nu) :~\exists f\in L^2(K,\mu),~\text{with } \int_{\mathcal{S}_K}  \eta \partial \phi  d\nu = \int_K f \phi d\mu,  \forall~\phi\in W^{1,2}(K) \right\},
\]
and we have $\partial^* \eta = f$. Observe that we have for $u\in\dom\Delta_K$,  $\partial^*\partial u = -\Delta_K u$. We then define the Hodge Laplacian by $\vec\Delta_K =- \partial\partial^*$ with the obvious domain. In the compact case, the domain of $\partial^*$ and $\vec{\Delta}_K$ can be described using spectral theory as shown in the next lemma.

\begin{lemma}\label{spectre}
We have
\[
\mathrm{dom}\, \partial^* =\left\{ \eta \in L^2(\mathcal{S}_K,\nu): ~\sum_{j=1}^{+\infty} \left( \int_{\mathcal{S}_K} \eta \partial \Phi_j d\nu \right)^2 <+\infty \right\},
\]
and for $\eta \in \mathrm{dom}\, \partial^*$,
\[
\partial^* \eta=\sum_{j=1}^{+\infty}  \left(\int_{\mathcal{S}_K} \eta \partial \Phi_j d\nu\right)   \Phi_j.
\]
Furthermore
\[
\dom\vec\Delta_K =\left\{ \eta \in L^2(\mathcal{S}_K,\nu):~\sum_{j=1}^{+\infty} \lambda_j \left(\int_{\mathcal{S}_K} \eta \partial \Phi_j d\nu\right)^2 < +\infty \right\}
\]
and for every $\eta \in \dom\vec\Delta_K $,
\[
-\vec\Delta_K \eta=\sum_{j=1}^{+\infty}   \left(\int_{\mathcal{S}_K} \eta \partial \Phi_j d\nu\right)  \partial \Phi_j.
\]

\end{lemma}

\begin{proof}
We observe first that
\begin{align*}
\dom\eng_K &=\left\{ f \in L^2(K,\mu):~\lim_{t \to 0} \left\langle \frac{f -e^{t \Delta_K}f}{t} , f \right\rangle_{L^2(K,\mu)} \text{  exists} \right\} 
\\&=
\left\{ f \in L^2(K,\mu):~\lim_{t \to 0} \frac{1}{t} \sum_{j=1}^{+\infty} (1-e^{- \lambda_j t}) \langle f , \Phi_j \rangle_{L^2(K,\mu)}^2 \text{  exists} \right\} 
\\&=
\left\{ f \in L^2(K,\mu):~\sum_{j=1}^{+\infty} \lambda_j \langle f, \Phi_j \rangle_{L^2(K,\mu)}^2 <+\infty \right\}  ,
\end{align*}
and moreover that for $f \in \dom \partial=\dom\eng_K$, 
\begin{align}\label{bar}
\partial f = \sum_{j=1}^{+\infty}  \langle f ,  \Phi_j \rangle_{L^2(K,\mu)} \partial \Phi_j.
\end{align}
As a consequence,
\[
\dom \partial^* =\left\{ \eta \in L^2(\mathcal{S}_K,\nu):~\sum_{j=1}^{+\infty}  \langle \eta,\partial \Phi_j \rangle_{L^2(\mathcal{S}_K,\nu)}^2 <+\infty \right\},
\]
and for $\eta \in \dom \partial^*$,
\[
\partial^* \eta=\sum_{j=1}^{+\infty}  \langle \eta , \partial \Phi_j \rangle_{L^2(\mathcal{S}_K,\nu)}  \Phi_j.
\]
From the definition of $\vec\Delta_K$, this immediately yields
 \[
\dom\vec\Delta_K =\left\{ \eta \in L^2(\mathcal{S}_K,\nu):~\sum_{j=1}^{+\infty} \lambda_j \langle \eta , \partial \Phi_j \rangle_{L^2(\mathcal{S}_K,\nu)}^2 < +\infty \right\}
\]
and for every $\eta \in \dom\vec\Delta_K $ we have,
\[
-\vec\Delta_K \eta=\sum_{j=1}^{+\infty}  \langle \eta , \partial \Phi_j \rangle_{L^2(\mathcal{S}_K,\nu)} \partial \Phi_j.
\]
\end{proof}

We  denote the Hodge semigroup  $e^{t\vec \Delta_K}$ by $\vec{P}^K_t$. The following result easily follows from Lemma \ref{spectre}.

\begin{theorem}
For $\eta  \in L^2(\mathcal S_K,\nu)$,
\[
\vec{P}^K_t \eta (x) =\int_{\mathcal S_K} \vec{p}_t^{\,\, K} (x,y) \eta (y) d\nu(y),
\]
where
\[
\vec{p}_t^{\,\, K} (x,y) =-\sum_{j=1}^{+\infty} \frac{1}{\lambda_j}e^{-\lambda_j t} \partial \Phi_j(x) \partial \Phi_j(y)
\]
is the Hodge heat kernel.
\end{theorem}

Finally, the following result is  analogous to Theorem \ref{bound hodge2} and can be proved in a similar way.

\begin{theorem}
There exist $c,C>0$ such that for every $t >0$ and $\nu$ a.e. $x,y \in \mathcal S_K$
    \[
 |\vec{p}_t^{\,\, K} (x,y)| \le  \frac{C e^{-\lambda_1 t} }{(1 \wedge t)^{1/d_w}}  \exp\biggl(-c\Bigl(\frac{d(x,y)^{d_{w}}}{t}\Bigr)^{\frac{1}{d_{w}-1}}\biggr).
 \]
\end{theorem}

\bibliographystyle{abbrv}

\noindent
Fabrice Baudoin: \url{fbaudoin@math.au.dk}\\
Department of Mathematics,
Aarhus University

\

\noindent Li Chen: \url{lchen@math.au.dk}\\
Department of Mathematics,
Aarhus University \& Department of Mathematics, Louisiana State University, Baton Rouge, LA 70803
\end{document}